\documentclass[12pt]{amsart}

\usepackage[usenames,dvipsnames]{color}

 \usepackage{latexsym,amsmath,amssymb,amsfonts,mathrsfs}
\usepackage{calc}
\usepackage{enumitem}
\usepackage[colorlinks=true]{hyperref}
 \usepackage[all]{xy}
\usepackage{widebar}
 \usepackage{defs,words,sets}

\newcommand{\marg}[1]{\footnote{#1}\marginpar[\hfill\tiny\thefootnote$\rightarrow$]{$\leftarrow$\tiny\thefootnote}}

\newcommand{\Jonathan}[1]{\marg{(Jonathan) #1}}

 \usepackage{notn}

\def\notndescdimen{2.5in}
\def\notnplace{\arabic{section}.\arabic{subsection}, p.\thepage}

\numberwithin{equation}{subsection}
\theoremstyle{plain}

\newtheorem{theorem}[equation]{Theorem}
\newtheorem{lemma}[equation]{Lemma}

\newtheorem{proposition}[equation]{Proposition}
\newtheorem{corollary}[equation]{Corollary}
\theoremstyle{remark}

\newtheorem{remark}[equation]{Remark}

\theoremstyle{definition}

\newcommand{\comment}[1]{}
\newcommand{\bbA}{\mathbb{A}}
\newcommand{\EE}{\mathbb{E}}
\newcommand{\FF}{\mathbb{F}}
\renewcommand{\LL}{\mathbb{L}}
\newcommand{\GG}{\mathbb{G}}
\renewcommand{\CC}{\mathbb{C}}
\renewcommand{\ZZ}{\mathbb{Z}}
\renewcommand{\QQ}{\mathbb{Q}}

\newcommand{\bA}{\mathbf{A}}

\newcommand{\bL}{\mathbf{L}}

\newcommand{\bZ}{\mathbf{Z}}

\newcommand{\eM}{\mathfrak{M}}

\newcommand{\fE}{\mathfrak{E}}
\newcommand{\fF}{\mathfrak{F}}

\newcommand{\fM}{\eM}
\newcommand{\tfM}{\widetilde{\fM}}

\newcommand{\fr}{\mathfrak{r}}
\newcommand{\sA}{\mathscr{A}}
\newcommand{\cB}{\mathcal{B}}

\newcommand{\sD}{\mathscr{D}}
\newcommand{\tD}{\widetilde{D}}
\newcommand{\sE}{\mathscr{E}}

\newcommand{\sL}{\mathcal{L}}
\newcommand{\sM}{\mathcal{M}}
\newcommand{\sN}{\mathcal{N}}

\newcommand{\cO}{\mathcal{O}}
\newcommand{\sP}{\mathcal{P}}

\newcommand{\sT}{\mathscr{T}}

\newcommand{\sX}{\mathscr{X}}
\newcommand{\tX}{\widetilde{X}}

\newcommand{\uI}{\underlinenormal{I}}
\newcommand{\oC}{\overlinenormal{C}}

\newcommand{\oE}{\overlinenormal{E}}

\newcommand{\oM}{\overlinenormal{M}}

\newcommand{\tsA}{\widetilde{\sA}}

\newcommand{\spec}{\mathrm{Spec}\;}

\newcommand{\rel}{\mathrm{rel}}
\newcommand{\orb}{\mathrm{orb}}
\newcommand{\relorb}{\mathrm{rel}}

\newcommand{\vir}{\mathrm{vir}}

\renewcommand{\Gm}{{{\bf G}_m}}
\newcommand{\BGm}{\cB\Gm}

\def\tototi{\mathbin{\mathop{\otimes}\limits^{\raise-1pt\hbox
{$\scriptscriptstyle {\rm L}$}}}}

\begin{document}
\title[Gromov--Witten invariants (\today)]{Relative and orbifold Gromov--{W}itten invariants}

\author[D. Abramovich]{Dan Abramovich}
 \thanks{Research of D.A. partially supported by NSF grants DMS-0335501, DMS-0603284, DMS-0901278, and DMS-1162367}
\address{Department of Mathematics, Box 1917, Brown University,
Providence, RI, 02912, U.S.A}
\email{abrmovic@math.brown.edu}
\author[C. Cadman]{Charles Cadman}
\thanks{Research of C.C. partially supported by NSF grant No. 0502170}
\address{Seattle, Washington, U.S.A}
\email{math@charlescadman.com}
\author[J. Wise]{Jonathan Wise}
\address{Univeristy of Colorado, Campus Box 395, Boulder, CO 80309-0395, U.S.A}
\email{jonathan.wise@colorado.edu}
\thanks{Research of J.W. partially supported by NSF post-doctoral fellowship MS-PRF 0802951 and NSA Young Investigator's Grant \#H98230-14-1-0107}
\date{\today}
\maketitle

\setcounter{tocdepth}{1}
\tableofcontents

\section{Introduction}

\subsection{Gromov-Witten invariants}
\label{sec:intro-invariants} Gromov--Witten invariants are deformation invariant numbers associated to a smooth variety $X$ over $\CC$ that are closely related to the numbers of curves in that variety with prescribed incidence to specified homology classes.  In this paper we focus on invariants associated to counting {\em rational} curves with prescribed tangencies along a fixed smooth divisor $D\subset X$, in addition to the prescribed incidence conditions. 

A seemingly mysterious phenomenon was observed in~\cite{CC}, where 
rational curves in the projective plane $X = \PP^2$ tangent to a smooth plane cubic $D$ were enumerated using the {\em orbifold Gromov--Witten invariants}  of root stacks $X_r$ branched along $D$. While it was observed that for low index $r$ the invariants exhibit erratic behavior, for large and divisible $r$ they stabilized.  Cadman and Chen compared these  with {\em relative Gromov--Witten invariants} of the pair $(X,D)$ computed earlier by Gathmann~\cite{Gath}.  It is no surprise that these orbifold and relative invariants agree when they are both enumerative: they count the same thing.  Remarkably, however, one is enumerative if and only if the other is, and the invariants coincide even when they are not enumerative.

Theorem ~\ref{maintheorem} below lays the question to rest in genus~$0$, showing that the rational orbifold invariants of  $X_r$ for large and divisible $r$ coincide with the rational relative invariants of $(X,D)$, for any $X$ and $D$.
We discuss the context of these invariants in Section \ref{Sec:context} below. Section \ref{sect:genus1} demonstrates that this genus-0 coincidence does not generalize. Indeed higher-genus invariants of $X_r$ do not stabilize, and thus cannot coincide with higher-genus relative invariants. A general correspondence between them remains to be discovered.

\subsection{Statement of the theorem}
\label{sec:thm-statement} 
We require the following input, standard in relative Gromov--Witten theory.  We fix
\begin{itemize}
\item  a smooth complex projective variety $X$, 
\item a smooth
divisor $D \subset X$ , 
\item a curve class $\beta\in H_2(X,\ZZ)$,
\item an $n$-tuple  of nonnegative integers $\mathbf{k} = (k_1,\ldots k_n)$, with $\sum k_i =
\beta\cdot D$, 
\item cohomology classes
$\gamma_1\ldots,\gamma_n$ where $\gamma_i\in H^*(X,\QQ)$ when $k_i =
0$ and $\gamma_i\in H^*(D,\QQ)$ when $k_i>0$, and
\item nonnegative integers
$a_1,\ldots,a_n$.
\end{itemize}
Finally, for $r$ a positive integer, we denote by $X_r= X_{D,r} = \sqrt[r]{X,D}$ the stack obtained by taking the $r$-th root of $X$ along $D$.

\begin{theorem}\label{maintheorem}
If $r$ is any sufficiently large and divisible natural number then the following relative and orbifold invariants coincide.
\begin{equation*}
\notn{\bigg\langle \prod_{i=1}^n\tau_{a_i}(\gamma_i,k_i)
\bigg\rangle^{(X,D)}_{0,\beta}}{relative GW invariant}
\ \ = \ \
\notn{\bigg\langle \prod_{i=1}^n\tau_{a_i}(\gamma_i,k_i)
\bigg\rangle^{X_r}_{0,\beta}}{orbifold GW invariant}
\end{equation*}
\end{theorem}

Our notation is explained in the following section.  There is also a table of notation in Appendix~\ref{app:notn}.

\subsection{Conventions} \label{Sec:conv}

\subsubsection{Relative stable maps} We use the moduli space
  \begin{equation*}
    \notn{\oM^{\rel}(X,D)}{moduli of relative stable maps}
    := \oM_{g,(k_1,\ldots,k_n)}^{\rel}(X,D,\beta)
  \end{equation*}
  of relative stable maps to $(X,D)$, where
\begin{itemize}
\item the source curve has genus $g$ and $n$ marked points,
\item the $i$-th marked point has contact order $k_i$ with $D$, and
\item the homology class of the curve is $\beta$.
\end{itemize}

Let $e_i^\rel$ be the $i$-th evaluation map, where 
\begin{align*}
  e_i^\rel & :\oM^\rel(X,D)\to X && \text{for $k_i = 0$, and} \\ 
  e_i^\rel & :\oM^\rel(X,D)\to D && \text{for $k_i > 0$}.
\end{align*}

Let $s_i: \oM^\rel(X,D)\to \oC$ be the $i$-th section of the universal {\em contracted curve mapping to $X$}, and let $\psi_i= c_1 s_i^*(\omega_{\oC/\oM^\rel(X,D)})$. The stack $\oM^{\rel}(X,D)$ admits a virtual fundamental class $\bigl[\oM^{\rel}(X,D)\bigr]^\vir$ defined in \cite{Li2}.

With this notation we set
\begin{equation*}
  \bigg\langle \prod_{i=1}^n\tau_{a_i}(\gamma_i,k_i)
  \bigg\rangle^{(X,D)}_{g,\beta}
  \ \ := \ \ \int_{\bigl[\oM^{\rel}(X,D)\bigr]^\vir} \psi_1^{a_1}e_1^*\gamma_1\cdots  \psi_n^{a_n}e_n^*\gamma_n.
\end{equation*}

\subsubsection{Orbifold stable maps} We use the moduli space
  \begin{equation*}
    \notn{\oM^{\orb}(X_r)}{moduli of orbifold stable maps}
    := \oM_{g,(k_1,\ldots,k_n)}(X_r,\beta)
  \end{equation*}
 of  stable maps to $X_r$, where
\begin{itemize}
\item the curve has genus $g$ and $n$ marked points,
\item the {\em coarse} evaluation map at the $i$-th marked point (defined below)
  \begin{equation*}
  e_i^\orb
    :\oM^{\orb}(X_r)\to \uI(X_r)
  \end{equation*}
  lands in the twisted sector of age $k_i/r$ (which is isomorphic to $X$ if $k_i=0$ and to $D$ if $k_i>0$), and
\item the homology class of the curve is $\beta$.
\end{itemize}

We have used the notation $\notn{\uI(X_r)}{coarse mod. sp. of inertia stack}$  for the coarse moduli space of the inertia stack
of $X_r$, which has $r$ components:
$$\uI(X_r)\cong X\sqcup D\sqcup\cdots\sqcup D.$$  The components isomorphic to $D$ are called twisted sectors, and are labeled by the ages $k_i/r \in [\:0,1)\cap \frac{1}{r}\bZ$. 

Let $s_i: \oM^\orb(X_r)\to \oC$ be the $i$-th section of the {\em universal coarse curve mapping to $X$}, and let $\psi_i= c_1 s_i^*(\omega_{\oC/\oM^\orb(X_r)})$. The stack $\oM^{\orb}(X_r)$ admits a virtual fundamental class  $\bigl[\oM^{\orb}(X_r)\bigr]^\vir$ defined in \cite{AGV-GW}.

With this notation we set
$$\left\langle \prod_{i=1}^n\tau_{a_i}(\gamma_i,k_i)
\right\rangle^{X_r}_{g,\beta}
\ \ := \ \ \int_{[\oM^{\orb}(X_r)]^\vir} \psi_1^{a_1}e_1^*\gamma_1\cdots  \psi_n^{a_n}e_n^*\gamma_n.$$

\subsection{Context}\label{Sec:context}

There are two essential ingredients in Gromov--Witten theory: a proper moduli space on which to do intersection theory, and a virtual fundamental class of the expected dimension in the homology of that moduli space.
As stated above, we are interested in counting rational curves in $X$ with prescribed incidence conditions, as well as prescribed tangencies along $D$.  
The introduction of tangencies makes these ingredients more subtle.  Since a tangency can degenerate to one of higher order, it is not obvious how to produce a proper moduli space, and since a tangency can be deformed to lower order, it is not obvious how to do the deformation theory necessary to produce a virtual fundamental class.

There are now several solutions to these problems.  The first is the theory of \emph{relative stable maps}, introduced by  A. M. Li and Y. Ruan in \cite{Li-Ruan}, with substantial contributions from Ionel--Parker \cite{Ionel-Parker1,Ionel-Parker2}, Gathmann \cite{Gathmann} and others.  Its algebro-geometric incarnation is due to J. Li~\cite{Li1,Li2}.  In this theory, tangencies are prevented from degenerating to higher order by allowing the target variety to expand, in close analogy to the way a Deligne--Mumford stable curve might expand to prevent a marked point from colliding with a node.  The deformation theory of maps from curves into expanded targets still remains quite subtle, however.

A second solution~\cite{cadman} is to change the target variety by a \emph{root construction}.  The variety is replaced by a stack that is isomorphic to the original variety away from the divisor, but in which the divisor is replaced by a ``stacky'' version of itself with a cyclotomic stabilizer group.  Provided that the order of the root construction is taken to be large enough, the concept of tangency to the divisor in the original variety can be replaced with transversal contact to the stacky divisor in the root stack.  Transversal contact is an open condition, so the ordinary theory of twisted stable maps~\cite{AV} applies to yield a proper moduli space and a virtual fundamental class via straightforward deformation theory.  The disadvantage of this theory, as compared to relative stable maps, is that it may include extraneous information in higher genus; see Section~\ref{sect:genus1}.

Cadman and Chen computed the orbifold invariants for rational curves with tangency to a smooth plane cubic~\cite{CC} and observed their agreement with Gathmann's earlier calculation of the relative invariants~\cite{Gath}, even when neither invariant is enumerative.  Thus we arrive at the coincidence observed by Cadman and Chen~\cite{CC}, already described in Section~\ref{sec:intro-invariants}.  Our general comparison result, Theorem~\ref{maintheorem}, explains this coincidence and generalizes it to rational curves in arbitrary targets.
%
%

\subsection{The intermediary}
Our comparison goes by way of a third theory, due to Abramovich and Fantechi~\cite{AF}, that combines the advantages of both while avoiding some of the disadvantages.  This ``relative--orbifold'' theory furnishes a correspondence between the relative and orbifold moduli spaces
\begin{equation*}
  \xymatrix{
    & \oM^{\relorb}(X_r, D_r) \ar[dr]^\Phi \ar[dl]_{\Psi} \\
    \oM^{\rel}(X, D) & & \oM^{\orb}(X_r), 
  }
\end{equation*}
see Section \ref{Sec:intermediate-space}. 

It is shown  in~\cite{AMW} that $\Psi$ is nearly an isomorphism%
---it is something like a root stack construction---and  that it identifies the virtual fundamental classes via pushforward, so our task is primarily to study the map $\Phi$.  This map is not an isomorphism, even in genus zero and for large $r$, but we will show that in genus zero and for a suitable choice of $r$ it nevertheless carries one virtual fundamental class to the other by push-forward, and therefore identifies the Gromov--Witten invariants.

\subsection{Other counting theories and comparisons}

There are at least two more theories of stable maps relative to a divisor, both based on the theory of logarithmic geometry.  The first of these was defined by B.\ Kim by introducing logarithmic structures into J.\ Li's degenerations~\cite{kim}.  Later work of Gross and Siebert~\cite{GS} and Abramovich and Chen~\cite{Chen,AC}   made it possible to study logarithmic stable maps without expanding the target variety.  The relationships between these theories, as well as with the original theory of J.\ Li and the orbifold theory of Abramovich and Fantechi, are treated in~\cite{AMW}.


\subsection{Counterexample in genus 1}
\label{sect:genus1}

Note that Theorem~\ref{maintheorem} applies only to genus zero invariants.  The necessity of this restriction is evident in the following example, which was shown to us by D.\ Maulik~\cite{Maulik}.  Let $E$ be an elliptic curve and let $X=E\times\PP^1$.  Let $D = X_0\cup X_{\infty}$, the union of the fibers of $X$ over $0$ and $\infty\in\PP^1$.  Let $f\in H_2(X)$ be the class of a fiber of $X\to\PP^1$.  Then the relative invariant with no insertions vanishes: $\langle\:\rangle^{(X,D)}_{1,f}=0$.  A simple explanation for this is that the invariant remains the same if we take a cover of $\PP^1$ branched at $0,\infty$, and at the same time it is multiplied by the degree of the cover.  Note that the space of genus~$1$ relative maps to $(X,D)$ of class $f$ has expected dimension $0$, even though the actual dimension is $1$.

Let $X_{r,s}$ be the stack obtained from $X$ via an $r$-th root construction on $X_0$ and an $s$-th root construction on $X_{\infty}$.  The space $\oM_{1,0}(X_{r,s},f)$ has a $1$-dimensional component and $r^2-1+s^2-1$ components of dimension $0$.  The $1$-dimensional component is isomorphic to the stack $\sP_{r,s}$, %
obtained from $\PP^1$ by an $r$-th root at $0$ and an $s$-th root at infinity.  The remaining components exist because a morphism $E\to E\times B\mu_r$ which is the identity onto the first factor is determined by a $\mu_r$-torsor over $E$.  There are $r^2$ choices for the $\mu_r$-torsor, but the trivial torsor already appears in the $1$-dimensional component.

The obstruction bundle on the $1$-dimensional component is the tangent bundle, which has degree $1/r+1/s$.  The $0$-dimensional components count with precisely their degree, which takes into account the automorphism group of the torsor.  We may therefore calculate
\begin{equation*}
  \langle\rangle^{\sX_{s,r}}_{1,f} = \frac{1}{r}+\frac{1}{s}+\frac{r^2-1}{r}+\frac{s^2-1}{s} = r+s.
\end{equation*}

We interpret this discrepancy between the relative and twisted Gromov--Witten invariants to be a result of the nontriviality of the Picard group of $E$.  We leave the formulation of a more precise statement to anyone interested in pursuing the question.




\subsection{Acknowledgements}

We gratefully acknowledge the help of Barbara Fantechi in understanding relative stable maps, Martin Olsson's help with $Hom$-stack, Davesh Maulik for a  crucial example in genus 1, and Angelo Vistoli with his insight on root stacks. In addition we thank Jarod Alper,
Linda Chen,
Alessio Corti,
Johan de Jong,
and Michael Thaddeus,
for helpful discussions at various stages of this project. 
Much progress was made while Abramovich and Wise were visiting MSRI in Spring 2009. We thank MSRI and the Algebraic Geometry program organizers for the opportunity afforded us to use its exciting environment. 

\section{Method of proof}

\subsection{An intermediate moduli space}\label{Sec:intermediate-space}

There is not a natural map in either direction between $\oM^\rel(X,D)$ and $\oM^\orb(X_r)$.  We will therefore require a third moduli space, in which both orbifold and relative geometry are present, to mediate between these moduli spaces.

We will use $\oM^\rel(X_r, D_r)$ as the intermediary (where $D_r$ is the $r$-th root divisor).  We denote by $e_i$ the $i$-th evaluation map, where for $k_i=0$ we have $e_i:\oM^\rel(X_r, D_r) \to X$ and for $k_i>0$ we have $e_i:\oM^\rel(X_r,D_r) \to D$.

Let $s_i: \oM^\rel(X_r,D_r)\to \oC$ be the $i$-th section of the universal \textit{coarse contracted curve} mapping to $X$, and let $\psi_i= c_1 s_i^*(\omega_{\oC/\oM^\rel(X_r,D_r)})$.

With this notation we set
\begin{equation*}
  \notn{\bigg\langle \prod_{i=1}^n\tau_{a_i}(\gamma_i,k_i)
    \bigg\rangle^{(X_r,D_r)}_{g,\beta}}{relative orbifold GW invariant}
  \ \ := \ \ \int_{\bigl[\oM^\rel_{g,n}(X_r,D_r,\beta)\bigr]^\vir} \psi_1^{a_1}e_1^*\gamma_1\cdots  \psi_n^{a_n}e_n^*\gamma_n.
\end{equation*}

%

\subsection{Reduction of main theorem to properties of virtual
  fundamental classes}
For any $g,r$ we have a diagram of stabilization morphisms
$$
\xymatrix{ & \oM^\rel(X_r,D_r)
  \ar[dl]_{\Psi} \ar[dr]^{\Phi}\\
 \oM^\rel(X,D) & &
 \oM^\orb(X_{r}).
 }
 $$

  We have defined the terms  so that
  \begin{itemize}
  \item $e_i^\rel\circ \Psi= e_i^{\relorb} = e_i^\rel\circ \Phi,$
  \item $\Psi^*\oC^\rel = \oC^\relorb = \Phi^*\oC^\orb,$ therefore
  \item $\Psi^*\omega_{\oC^\rel/\oM^\rel} = \omega_{\oC^\relorb/\oM^\relorb} = \Phi^*\omega_{\oC^\orb/\oM^\orb},$ and finally
  \item $\Psi^*s_i^\rel = s_i^\relorb = \Phi^*s_i^\orb.$
  \end{itemize}

  Consequently, the projection formula gives us
\begin{align*}
\bigg\langle \prod_{i=1}^n\tau_{a_i}(\gamma_i,k_i)
\bigg\rangle^{(X_r,D_r)}_{0,\beta}
& =  \int_{\Psi_*([\oM^\rel(X_r,D_r)]^\vir)} \psi_1^{a_1}e_1^*\gamma_1\cdots  \psi_n^{a_n}e_n^*\gamma_n\\
& =  \int_{\Phi_*([\oM^\rel(X_r,D_r)]^\vir)} \psi_1^{a_1}e_1^*\gamma_1\cdots  \psi_n^{a_n}e_n^*\gamma_n
\end{align*}
where the integrals are on $\oM^\rel(X,D)$ and $\oM^\orb(X_r)$, respectively.
Theorem~\ref{maintheorem} is thus a consequence of the following two theorems.


\begin{theorem} \label{Prop1}
For all $g$, $\beta$, and $r$,
\begin{equation*}
    \Psi_*\Bigl(\bigl[\oM^\relorb(X_r,D_r)\bigr]^{\vir}\Bigr)  = \bigl[\oM^{\rel}(X,D)\bigr]^{\vir}  .
\end{equation*} 
\end{theorem}

\begin{theorem}\label{Prop2}  If $g=0$ and $r$ is sufficiently large and divisible (depending on $\beta$) we have
\begin{equation*}
\Phi_*\Bigl(\bigl[\oM^\relorb(X_r,D_r)\bigr]^{\vir}\Bigr) = \bigl[\oM^\orb(X_r)\bigr]^{\vir} .
\end{equation*}
\end{theorem}


We will prove Theorems~\ref{Prop1} and~\ref{Prop2} using a technique introduced by Costello.

\subsection{Costello's diagrams}

The applications of Costello's method to the proofs of Theorems~\ref{Prop1} and~\ref{Prop2} are broadly similar, so in order to fix ideas we will refer only Theorem~\ref{Prop2} in this part of the summary.

We restrict to genus $0$ maps and construct a cartesian square
\begin{equation}\label{Eq:Costello}
\vcenter{\xymatrix@C=5em{
\oM^{\relorb}_{g=0}(X_r,D_r) \ar^<>(0.5){\Phi_X}[r]\ar_{\sigma^\rel}[d]
 & \oM^\orb_{g=0}(X_r)\ar[d]^{\sigma} \\
 \fM^{\rel}_{g=0}(\sA_r, \sD_r)' \ar^<>(0.5){\Phi_\sA}[r] & \fM^{\orb}_{g = 0}(\sA_r)'. }}
\end{equation}
$\mnotn{\fM^{\rel}_{g=0}(\sA_r,\sD_r)'}{special open subset of $\fM^{\rel}_{g=0}(\sA,\sD)$}$
$\mnotn{\fM^{\orb}_{g=0}(\sA_r)'}{special open subset of $\fM^{\orb}_{g=0}(\sA)$}$

Here $\Phi_X$ is the morphism denoted $\Phi$ previously.  We use $\notn{\sA}{moduli of line bundle with section}$ to stand for the stack $[\bbA^1/\GG_m]$.
  We may view $\sA$ as the moduli stack of pairs $(L, s)$ where $L$ is a line bundle and $s$ is a section of $L$.  The vanishing locus of $s$ is a divisor $\notn{\sD}{universal Cartier divisor}$, isomorphic as a stack to $\BGm$.  Note that the $r$-th root construction $(\sA_r, \sD_r)$ is abstractly isomorphic to $(\sA, \sD)$; we retain the subscript to emphasize that the map $\sA_r \rightarrow \sA$ is not the identity.\footnote{In fact, $\sA_r \rightarrow \sA$ is the \textit{universal} $r$-th root construction.}

We show below that the diagram has the following properties:
\begin{enumerate}
	\item the virtual fundamental classes can be defined via perfect relative obstruction theories relative to the vertical arrows, and
	\item the hypotheses of \cite[Theorem~5.0.1]{costello} are satisfied for these choices.
\end{enumerate}

The stacks in the bottom row of this diagram require definitions, which will be given explicitly in Sections~\ref{sec:universal} and ~\ref{sec:conclusion}.  For the moment, we note the following:

\begin{itemize}
\item
We define
$\notn{\fM(\sA)}{curves with line bundle and section}$ to be the stack of triples $(C,L,s)$, where $C$ is a twisted curve, $L$ a line bundle on $C$ and $s$ a section of $L$, such that the associated morphism $C \to \sA$ is representable. The stack $\fM_{g=0}(\sA)'$ is defined in Section~\ref{sec:orb-dense-open} as an open substack of $\fM_{g=0}(\sA)$.
\item
We define $\notn{\fM^\rel(\sA, \sD)}{curves with line bundle, section, and expansion}$ to be the stack of pre-stable maps to $\sA$, relative to the divisor $\sD$.  The stack $\fM^\rel_{g=0}(\sA, \sD)'$ is defined in Section~\ref{sec:conclusion} and is \'etale over $\fM^\rel_{g=0}(\sA, \sD)$.
\end{itemize}

We regard the morphism $\Phi_{\sA}$ as the \emph{universal example} of the arrow $\Phi_X$.  In Section~\ref{sec:universal} we will show that $\Phi_{\sA}$ is birational, which we regard as the universal case of Theorem~\ref{Prop2}.  This is the technical heart of the paper, and the place where the restriction to genus~$0$ and the requirements on $r$ are needed.

In Section~\ref{sec:obs} we show that the virtual fundamental classes of $\oM^{\rm rel}_{g=0}(X_r,D_r)$ and $\oM^{\rm orb}_{g=0}(X_r)$ may be defined via obstruction theories relative to $\fM^{\rm rel}_{g=0}(\sA_r,\sD_r)'$ and $\fM^{\rm orb}_{g=0}(\sA_r)'$, respectively, and in Section~\ref{sec:proof2} we show that these relative obstruction theories are compatible.  Combined with the universal case, this will suffice to verify the hypotheses of Costello's theorem~\cite[Theorem~5.0.1]{costello} for diagram~\eqref{Eq:Costello} and imply Theorem~\ref{Prop2}.

\section{The moduli spaces}
\label{sec:spaces}

\subsection{Smooth pairs}

A smooth pair is a pair $(X,D)$ where $X$ is a smooth algebraic stack and $D$ is a smooth divisor on $X$.  A morphism of smooth pairs $(X,D) \rightarrow (Y, E)$ is a morphism $f : X \rightarrow Y$ such that $f^{-1} E = D$.  There is a universal example of a smooth pair:  $(\sA, \sD)$ where $\sA = [ \bA^1 / \Gm ]$ and $\sD = [ \: 0 \: / \: \Gm] \subset \sA$.  If $(X,D)$ is another smooth pair then there is a unique morphism $f : X \rightarrow \sA$ such that $D = f^{-1}(\sD)$.  

We can interpret $\sA$ as the moduli space of pairs $(L,s)$ where $L$ is a line bundle and $s$ is a section of $L$.  As such there is, for each non-negative integer $r$, a map $[r] : \sA \rightarrow \sA$ sending $(L,s)$ to $(L^{\tensor r}, s^r)$.  We will sometimes write $\sA_r$ for the source of $[r]$ in order to emphasize the map to $\sA$.  We write $\sD_r$ for the universal divisor $\sD$ under this identification.  We therefore have a map of pairs $(\sA_r, \sD_r) \rightarrow (\sA, \sD)$.  We view $(\sA_r, \sD_r)$ as the $r$-th root of $(\sA, \sD)$; it is thus the \emph{universal $r$-th root construction}.

Given a smooth pair $(X,D)$ we can form the fiber product $X_r = X \fp_{\sA} \sA_r$ where the map $X \rightarrow \sA$ is the one associated to the divisor $D$.  We write $D_r$ for the pre-image of $\sD_r$ under the map $X \rightarrow \sA_r$.  Then $(X_r, D_r)$ is the \emph{root stack} of $X_r$ along $D_r$.

\subsection{Orbifold stable maps}

Let $\sX$ be a smooth algebraic stack.  Following \cite{AV}, we can define a moduli space $\fM^\orb(\sX)$ of orbifold pre-stable maps into $\sX$ whose $S$-points are diagrams
\begin{equation*} \xymatrix{
C \ar[r]^f \ar[d]_\pi & \sX \\
S
} \end{equation*}
where
\begin{enumerate}[label=(\roman{*})]
\item $C$ is an orbifold pre-stable curve over $S$,
\item $f$ is representable.
\end{enumerate}
When $\sX$ is a Deligne--Mumford stack, we also define $\oM(\sX)$ to be the open substack of $\fM(\sX)$ consisting of those diagrams that satisfy the stability condition:
\begin{enumerate}[resume*]
\item (stability) fiberwise over $S$, the automorphism group of $C$ over $\sX$ is finite.
\end{enumerate}

Orbifold pre-stable maps to the target $\sA$ will be of particular importance to us.  We may interpret $\fM(\sA)$ as the moduli space whose $S$-points are triples $(C, L, s)$ where $C$ is an orbifold pre-stable curve over $S$, $L$ is a line bundle on $C$, and $s$ is an element of $\Gamma(C, L)$.

\subsection{Relative stable maps}

In this paper, we will require relative stable maps to orbifold targets. This requires only a slight modification to J.\ Li's original definitions.  The definition relies on the notion of an expanded pair, for whose definition we refer the reader to \cite[Section~2.1]{ACFW}.

We write $\sT$ for the moduli space of expansions of the pair $(\sA, \sD)$, which is, by definition~\cite[Definition~2.1.6]{ACFW}, the moduli space of expansions of any pair $(X,D)$.  We write $(\sA^{\exp},\sD^{\exp})$ for the universal expansion of $\sA$ and $(X^{\exp}, D^{\exp})$ for the universal expansion of $(X,D)$.

Given an expansion $\tsA_r$ of $\sA_r$ over a base $S$ we may obtain an expansion of $\sA$ by passing to the relative coarse moduli space of the morphism $\tsA_r \rightarrow \sA \times S$.  This gives a morphism $\sT \rightarrow \sT$.  In order to emphasize that this morphism is not the identity, we employ the notation $\sT_r$ for its source.  Using this notation, we have a commutative diagram
\begin{equation*} \xymatrix{
\sA_r^{\exp} \ar[r] \ar[d] & \sA^{\exp} \ar[d] \\
\sT_r \ar[r] & \sT .
} \end{equation*}
Note:  this diagram is not cartesian!  See \cite[Section~7]{ACFW} for more about these untwisting morphisms.

Let $(X,D)$ be a smooth pair.  A pre-stable relative map to $(X,D)$ over $S$ consists of
\begin{enumerate}
\item an expansion $(\tX, \tD)$ of the pair $(X,D)$ over $S$,
\item a pre-stable orbifold curve $C$ over $S$, and
\item an $S$-morphism $f : C \rightarrow \tX$
\end{enumerate}
subject to the predeformability condition
\begin{enumerate}[resume*]
\item (predeformability) for any node of $C$ that maps to the singular locus of $\tX / S$ there are \'etale-local coordinates near the node in $C$, and smooth-local coordiantes near its image in $\tX$, such that $f$ has the following local form:
\begin{equation*} \xymatrix@R=0pt{
\mathcal{O}_S[x,y] / (xy - t) & \ar[l] \mathcal{O}_S[u,v] / (uv - t^r) \\
x^r & \ar@{|->}[l] u \\
y^r & \ar@{|->}[l] v
} \end{equation*}
\end{enumerate}
The stack of relative pre-stable maps to $(X,D)$ is denoted $\fM^{\rel}(X,D)$.  We say that a pre-stable relative map is stable if it satisfies the following condition:
\begin{enumerate}[resume*]
\item (stability) the automorphisms of $f : C \rightarrow \tX$ compatible with the projection to $X$ are finite, when viewed as a group scheme over $S$.
\end{enumerate}
Note that automorphisms of $f$ are commutative diagrams
\begin{equation*} \xymatrix{
C \ar[r] \ar[d] & \tX \ar[d] \\
C \ar[r]  & \tX
} \end{equation*}
where $C \rightarrow C$ is an automorphism of $C$ as an orbifold pre-stable curves and $\tX \rightarrow \tX$ is an automorphism of $\tX$ as an expansion of $(X,D)$.

We note that by forgetting the curve, we have a morphism of stacks $\fM^\rel(X,D) \rightarrow \sT$.  These fit into a commutative diagram:
\begin{equation*} \xymatrix{
\fM^\rel(X_r,D_r) \ar[r] \ar[d] & \fM^\rel(X,D) \ar[d] \\
\sT_r \ar[r] & \sT 
} \end{equation*}

\begin{theorem}
The stack $\oM^{\rel}(X_r, D_r)$ is proper and of Deligne--Mumford type.
\end{theorem}
\begin{proof}
See \cite[Theorem~2.2.1]{AF2}.
\end{proof}

\section{The universal case}
\label{sec:universal}

Here we treat analogues of Theorems~\ref{Prop1} and~\ref{Prop2} where $(X,D)$ is replaced by $(\sA, \sD)$.  Strictly speaking, the results we prove here are not special cases of Theorems~\ref{Prop1} and~\ref{Prop2} as we work with pre-stable maps here instead of stable maps.  The results proved in this section are the essential input in our later application of Costello's theorem.

\subsection{Relative maps}

We will say that an object $f : C \rightarrow \tsA$ of $\fM^{\rel}(\sA, \sD)$ is \textit{totally non-degenerate} if the following conditions hold:
\begin{enumerate}
\item the target is unexpanded (meaning the contraction $\tsA \rightarrow \sA$ is an isomorphism), so that the object lies in the open substck $\fM(\sA)$,
\item $C$ is smooth, and
\item $f(C)$ is not contained in $\sD$.
\end{enumerate}
Remarking that $\fM^{\rel}(\sA_r, \sD_r)$ is abstractly isomorphic to $\fM^{\rel}(\sA, \sD)$ we also speak of totally nondegenerate objects of $\fM^{\rel}(\sA_r, \sD_r)$.

The following is the universal analogue of Theorem~\ref{Prop1}:

\begin{theorem} \label{thm:rel-birational}
The map $\fM^\rel(\sA_r, \sD_r) \rightarrow \fM^\rel(\sA, \sD)$ is birational.
\end{theorem}

It is immediate that this map induces an isomorphism between the loci of totally non-degenerate objects.  It therefore suffices to show that these loci are dense.  Since the pairs $(\sA_r, \sD_r)$ and $(\sA, \sD)$ are abstractly isomorphic, the following lemma implies the theorem.

\begin{lemma} \label{lem:rel-dense-open}
The totally non-degenerate objects in $\fM^{\rel}(\sA,\sD)$ are dense.
\end{lemma}
\begin{proof}
Let $C \rightarrow \tsA$ be a relative stable map to an expansion of $\sA$.  By induction, it will be sufficient to show that one of the nodes of $\tsA$ can be smoothed.  Let $\tsA' = \sA \amalg_{\sD} \sA$ where the two copies of $\sA$ are joined together along the automorphism of $\sD$ sending a line bundle to its dual.  We can find a map $\tsA \rightarrow \tsA'$ that is an isomorphism near any given node of $\tsA$ and collapses every other point of $\tsA$ to one or the other point of $\tsA'$.  Taking advantage of the section $\tsA' \rightarrow \tsA$, which has open image, it is easy to extend a deformation of $\tsA'$ to $\tsA$.  We can replace $\tsA$ with $\tsA'$ for the rest of the proof and assume that $\tsA$ has just two irreducible components.

It is sufficient to produce the desired deformation in an \'etale neighborhood of the pre-image of the node of $\tsA$.  Indeed, once an infinitesimal deformation is found in a neighborhood of the pre-image of the node, one only needs to observe that local deformations of curves can always be glued, and away from the nodes in the pre-image of the node of $\tsA$, we are merely gluing together maps to a point, which is a trivial matter.

Now, working \'etale locally we can assume that the pre-image of the node of $\tsA$ is a disjoint union of copies of $U_i = \Spec \CC[x_i,y_i] / (x_i y_i)$; we can assume moreover that each of the maps $U_i \rightarrow \tsA$ factors through the smooth cover $\Spec \CC[u_i,v_i] / (u_i v_i) \rightarrow \tsA$ in the form $u_i \mapsto x_i^{r_i}$, $v_i \mapsto y_i^{r_i}$.  Now let $r = \lcm \set{r_i}$.  Over the base $\CC[t]$ we now have the maps
\begin{equation*} 
\Spec \CC[x_i,y_i,t] / (x_i y_i = t^{r/ r_i}) \rightarrow  \Spec \CC[u,v] / (uv = t^r)
\end{equation*}
extending the given ones.  This clearly gives a smoothing.
\end{proof}

\subsection{Orbifold maps}
\label{sec:orb-dense-open}

Theorem \ref{thm:orb-birational}, below, is the birationality statement for the map $\fM^{\rel}(\sA,\sD) \rightarrow \fM^{\orb}(\sA)$, and thus also of the map  $\fM^{\rel}(\sA_r,\sD_r) \rightarrow \fM^{\orb}(\sA_r)$.  It is more subtle than the one considered in the last section.  In fact, we will only obtain the required statement in genus~$0$ and for a non-dense open substack of $\fM^{\orb}_{g=0}(\sA)$.  This will still suffice for our eventual purposes, since the open substack in question will contain the image of the map $\oM^{\orb}_{g=0}(X_r)$, at least when $r$ is sufficiently large and divisible.  These issues explain the restrictions in Theorem~\ref{maintheorem}.

Let $\fM^\orb_{g=0}(\sA)'$ be the open substack of $\fM^\orb_{g=0}(\sA)$ parameterizing triples $(C,L,s)$ such that
\begin{enumerate}
\item \label{cond:1} $\deg L = \sum_x \age_x(L)$, the sum taken over the smooth points of $C$, and
\item \label{cond:2} for each proper subcurve $D \subset C$, we have $-\frac{1}{2} < \deg(L \rest{D}) < \frac{1}{2}$.
\end{enumerate}
Let $\fM^{\rel}_{g=0}(\sA, \sD)''$ be the pre-image of $\fM^{\orb}_{g=0}(\sA)'$ under the projection $\fM^{\rel}(\sA, \sD) \rightarrow \fM^{\orb}(\sA)$.\footnote{The stack $\fM^{\rel}_{g=0}(\sA,\sD)''$ is very close to, but not exactly the same as, the stack $\fM^{\rel}_{g=0}(\sA,\sD)'$ appearing in diagram~\ref{Eq:Costello}.  A small adjustment will be required in Section~\ref{sec:proof2}.}

\begin{remark}
Note that if $(C, L, s)$ is a point of $\fM^{\orb}_{g=0}(\sA)'$ where $C$ is an ordinary curve (i.e., has no orbifold points) then Conditions~\ref{cond:1} and~\ref{cond:2} together with $g=0$ imply that $L$ is trivial, rendering the locus of non-orbifold curves in $\fM^{\orb}_{g=0}(\sA)'$ entirely uninteresting.
\end{remark}

The universal analogue of Theorem~\ref{Prop2} is the following:
\begin{theorem} \label{thm:orb-birational}
The map $\Psi : \fM^{\rel}_{g=0}(\sA, \sD)'' \rightarrow \fM^{\orb}_{g=0}(\sA)'$ is birational.
\end{theorem}

As before, we will show that the totally non-degenerate objects are dense on source and target, where an object $(C,L,s)$ of $\fM^{\orb}(\sA)$ is called totally non-degenerate if $C$ is smooth and $s$ does not vanish identically.  It is immediate that $\Psi$ restricts to an isomorphism on the totally nondegenerate objects, and we have already seen in Lemma~\ref{lem:rel-dense-open} that totally non-degenerate objects are dense in $\fM^{\rel}_{g=0}(\sA, \sD)''$, so the theorem reduces to the following lemma:

\begin{lemma} \label{lem:non-degen-dense-in-M-orb}
The totally non-degenerate objects in $\fM^{\orb}_{g = 0}(\sA)'$ are dense.
\end{lemma}

\subsubsection{Proof of Lemma~\ref{lem:non-degen-dense-in-M-orb}}

Our strategy is to filter $\fM^{\orb}_{g=0}(\sA)'$ by open subsets and show that each is dense in the next.  Consider
\begin{equation*}
U_0 \subset U_1 \subset \cdots \subset \fM^{\orb}_{g=0}(\sA)'
\end{equation*}
where $U_n$ is the locus of triples $(C,L,s)$ where $C$ has at most $n$ nodes.  Thus $U_0$ is the locus of $(C,L,s)$ with $C$ smooth.  We have $\bigcup U_n = \fM^{\orb}_{g=0}(\sA)'$, so the proof of Lemma~\ref{lem:non-degen-dense-in-M-orb} reduces to the verification of the following two lemmas:

\begin{lemma}\label{lem:non-degen-dense-in-smooth}
The totally nondegenerate triples $(C,L,s)$ are dense in $U_0$.
\end{lemma}

\begin{lemma} \label{lem:Un-dense-in-next}
For each $n$, the open substack $U_n$ is dense in $U_{n+1}$.
\end{lemma}

\begin{proof}[Proof of Lemma~\ref{lem:non-degen-dense-in-smooth}]
We must show that if $s$ vanishes identically then it can be deformed not to vanish identically.  But Conditions~\ref{cond:1} and~\ref{cond:2} combine here to imply that $0 \leq \deg L < 1/2$.  Therefore, if $\pi : C \rightarrow \oC$ denotes the projection to the coarse moduli space, $\pi_\ast L$ is trivial.  We may therefore view $s$ as a section of the trivial line bundle on $\oC$, which can clearly be deformed to a nowhere vanishing section.  The corresponding section of $L$ will be generically nonzero.
\end{proof}

The following lemmas will be useful in our proof of Lemma~\ref{lem:Un-dense-in-next}. 

\begin{lemma} \label{lem:vanishing-at-node}
Suppose $(C, L, s)$ is a point of $\fM^{\orb}_{g=0}(\sA)'$ and $x$ is a node of $C$ such that $s(x) = 0$.  
\begin{enumerate}
\item If $L \rest{x}$ is trivial then $s$ vanishes identically on both branches of $C$ containing $x$.
\item If $L \rest{x}$ is non-trivial then $s$ vanishes identically on at least one branch of $C$ containing $x$.
\end{enumerate}
\end{lemma}
\begin{proof}
For the first assertion, we note that if $L \rest{x}$ is trivial then $C=\oC$ near $x$ by the representability of the map $C \rightarrow \sA$; if $s$ does not vanish identically on an irreducible component $D$ of $C$ containing $x$ then $\deg L \rest{D} \geq 1$, contradicting Condition~\ref{cond:2}.

For the second assertion, assume that $D$ and $E$ are the components of $C$ meeting at $x$.  Then $\age_x L \rest{D}  + \age_x L \rest{E}  = 1$.  Therefore one of these---say $\age_x L \rest{D}$---must be $\geq 1/2$.  But once again, if $s$ does not vanish identically on $D$, then
\begin{equation*}
\deg L \rest{D} \geq \age_x L \rest{D} \geq 1/2,
\end{equation*}
again contradicting Condition~\ref{cond:2}.
\end{proof}

For a possibly orbifold point $x$ on a curve $E$ we write $r_x$ for its index, so that $\deg x = 1/r_x$. If $L$ is a line bundle on $E$, then $0 \leq r_x \age_xL < r_x$ is an integer.
\begin{lemma} \label{lem:structure-of-Ls}
Suppose that $(C,L,s)$ is a point of $\fM^{\orb}_{g=0}(\sA)'$, that $E$ is an irreducible component of $C$.
\begin{enumerate}[label=(\roman{*})]
\item \label{lem:structure-of-Ls:1} If $s$ does not vanish identically on $E$ then 
\begin{equation*}
L \rest{E} \simeq \cO_E\Bigl( \sum_{x \in E} r_x\age_x \bigl( L \rest{E} \bigr) x \Bigr)  .
\end{equation*}
\item \label{lem:structure-of-Ls:2} If $s$ vanishes identically on $E$ but does not vanish identically on any other irreducible component of $C$ meeting $E$ then
\begin{equation*} 
L \rest{E} \simeq \cO_E\Bigl( \sum_{x \in E} r_x\age_x \bigl( L \rest{E} \bigr) x \Bigr) \tensor \cO_E(- \bigl| E \cap C^{\rm sing} \bigr| ) .
\end{equation*}
\end{enumerate}
\end{lemma}
\begin{proof}
We certainly have 
\begin{equation*}
L\rest{E} \simeq \cO_E \Bigl( \sum_{x \in E} r_x\age_x \bigl( L \rest{E} \bigr) x \Bigr) \tensor \cO_E(d)
\end{equation*}
for some integer $d$.  

Consider first the case where $s$ does not vanish identically on $E$.  Let $\pi : E \rightarrow \oE$ denote the coarse moduli space map.  Then $\pi_\ast L \simeq \cO_{\oE}(d)$.  If $s$ does not vanish identically on $E$ then $\pi_\ast s$ is a non-zero section of $\cO_{\oE}(d)$, whence $d \geq 0$.  On the other hand, $\deg L \rest{E} < 1/2$ by Condition~\ref{cond:2}.  It follows that $d = 0$.

Now assume that $s \rest{E} = 0$.  Let $D = \overlinenormal{C \smallsetminus E}$ so $E \cap D = E \cap C^{\rm sing}$.  We have
\begin{equation*}
\sum_{x \in E \cap C^{\rm sm}} \age_x(L \rest{E}) < \sum_{x \in C^{\rm sm}} \age_x(L) = \deg L < 1/2 .
\end{equation*}
On the other hand,
\begin{align*}
\sum_{x \in E \cap C^{\rm sing}} \age_x(L \rest{E}) & = \sum_{x \in E \cap C^{\rm sing}} \bigl(1 - \age_x(L \rest{D})\bigr)  \\
& = \bigl| E \cap C^{\rm sing} \bigr| - \sum_{x \in E \cap C^{\rm sing}} \age_x (L \rest{D})
\end{align*}
But, by assumption, $s$ does not vanish identically on any component of $D$ meeting $E$, so that
\begin{equation*}
\sum_{x \in E \cap C^{\rm sing}} \age_x (L \rest{D}) \leq \deg L \rest{D} < 1/2 ,
\end{equation*}
again by Condition~\ref{cond:2}.  Thus
\begin{gather*}
\bigl| E \cap C^{\rm sing} \bigr| - 1/2 < \sum_{x \in E \cap C^{\rm sing}} \age_x(L \rest{E}) \leq \big| E \cap C^{\rm sing} \bigr| \\
0 \leq \sum_{x \in E \cap C^{\rm sm}} \age_x(L \rest{E}) < 1/2 .
\end{gather*}
Combining these we discover $d = \: -  \bigl| D \cap E \bigr|$.
\end{proof}

\begin{proof}[Proof of Lemma~\ref{lem:Un-dense-in-next}]\hfill

\begin{enumerate}[wide,
label=\textsc{Step~\arabic{*}.},
ref=\textsc{Step~\arabic{*}}]
\item \label{step:1}

If $C$ contains a node at which $s$ vanishes on both branches, we choose any $1$-parameter smoothing $C'$ of that node that does not smooth any other node.  Since deformations of line bundles on curves are unobstructed, we may extend $L$ to be a line bundle $L'$ on $C'$.  We extend $s$ to be a section that vanishes on the component of $C'$ that specializes to contain the given node.

\item \label{step:3} After \ref{step:1}, we may assume that if $s$ vanishes on an irreducible component $E$ of $C$ then $s$ does not vanish on any irreducible component of $C$ intersecting $E$.  Taking $E$ to be an irreducible component of $C$ on which $s$ vanishes, we shall smooth all of the nodes of $C$ lying on $E$ simultaneously.

We construct a $1$-parameter smoothing $C'$ of $C$ such that $L$ and $s$ extend to a line bundle $L'$ on $C'$ and a section $s'$ of $L'$.  It will be necessary to construct $C'$ with the appropriate local structure near the nodes of $C$.  Fortunately, deformations of nodal curves are unobstructed so every choice of local deformation lifts to a global deformation.  We may therefore select the local structure of $C'$ as we please.

Let $S' = \Spec \CC[[t^{1/A}]]$, where $A$ is a positive integer to be determined later.  Suppose that $\xi$ is a node of $C$ with $\mu_n$-orbifold structure.  Let $D \subset C$ be the irreducible component meeting $\xi$ where $s$ does not vanish identically.  We may select local coordinates $x^{1/n}$ on $D$ and $y^{1/n}$ on $E$ at the point $\xi$ with $x^{1/n} y^{1/n} = 0$.  We can represent $(L, s)$ locally by $(\mathcal{O}_{C}, x^{a})$ where $a = \age_{P}(L \rest{D})$.  Recall that $a \in [0,1) \cap \mathbf{Q}$ and $na$ is an integer.  We give $C'$ the local structure $x^{1/n} y^{1/n} = t^{1/na}$ near $\xi$.  Having done this near every node $\xi$ of $C$ that lies on $E$, we take $A$ to be a positive integer such that $A$ is divisible by all of the integers $na$ described above.

Now we define a Cartier divisor giving $L'$ and $s'$ on the deformation $C'$.  We begin by noting that for every irreducible component $E$ of $C$ where $s$ vanishes we may construct a Cartier divisor on $C'$ as follows:  near the node $\xi$ described in the last paragraph, take the Cartier divisor defined by $x^{a}$ and note that this agrees with the Cartier divisor defined by $t$ in a neighborhood of any smooth point of $E$ and is empty near any point of $C' \smallsetminus E$.  We denote this divisor by $E'$.  Let $F$ be the union of $E$ and all irreducible components of $C$ meeting $E$.  Define $L'$ to be the line bundle that agrees with
\begin{equation*}
\mathcal{O}_{C'}\Bigl(E' + \sum_{x \in F^{\rm sm}} r_x\age_x \bigl( L \rest{E} \bigr)\; x \Bigr)
\end{equation*}
on $F$ and with $L$ on $\overline{C \smallsetminus F}$ (the components of $C$ not affected by the smoothing).  Let $s'$ be the section associated to the function $1$ on $F$ and take $s'$ to agree with $s$ where $L'$ agrees with $L$.

To complete \ref{step:3}, we need to show that $(L' \rest{C}, s' \rest{C}) \simeq (L, s)$.  Since the genus of $C$ is zero, it is sufficient to verify that $(L' \rest{C_0}, s' \rest{C_0}) \simeq (L \rest{C_0}, s \rest{C_0})$ for every irreducible component $C_0$ of $C$.  

When $C_0$ is not contained in $F$, this is true by definition.  When $C_0$ is one of the irreducible components of $C$ meeting $E$ in a node $x$ note that $\mathcal{O}_{C_0}(E') = \mathcal{O}_{C_0}(r_x \age_x(L \rest{C_0}) x)$.  Combining this with Lemma~\ref{lem:structure-of-Ls}~\ref{lem:structure-of-Ls:1} gives $L' \rest{C_0} \simeq L \rest{C_0}$.

Finally, when $C_0 = E$, let $D = \overlinenormal{C \smallsetminus E}$.  We can represent $\mathcal{O}_E(E')$ by the divisor 
\begin{multline*}
\sum_{x \in E \cap C^{\rm sm}} r_x\age_x \bigl( L \rest{E} \bigr) x \: - \sum_{x \in E \cap C^{\rm sing}} r_x\age_x \bigl( L \rest{D} \bigr) \; x \\
= \sum_{x \in E \cap C^{\rm sm}} r_x\age_x \bigl( L \rest{E} \bigr) x \: + \sum_{x \in E \cap C^{\rm sing}} r_x\age_x \bigl( L \rest{E} \bigr) \; x \: - \sum_{x \in E \cap C^{\rm sing}} r_x \; x \\
= \sum_{x \in E} r_x\age_x \bigl( L \rest{E} \bigr) \; x \: - \big| E \cap C^{\rm sing} \bigr|
\end{multline*}
which agrees with $L \rest{E}$ by Lemma~\ref{lem:structure-of-Ls}.

\item \label{step:2} We may now assume that there is no irreducible component of $C$ on which $s$ vanishes identically.  Then by Lemma~\ref{lem:vanishing-at-node}, $s$ cannot vanish at any node of $C$.  Therefore we must have
\begin{equation*}
(L,s) \simeq \Bigl(\cO_C\bigl(\sum_{x \in C^{\rm sm}} r_x\age_x(L) \; x \bigr), 1\Bigr).
\end{equation*}
Choose a $1$-parameter smoothing $C'$ of $C$.  Then
\begin{equation*}
(L',s') = \Bigl(\cO_{C'}\bigl(\sum_{x \in {C'}^{\rm sm}} r_x\age_x(L) \; x \bigr), 1 \Bigr)
\end{equation*}
extends $(L,s)$ to $C'$.
\end{enumerate}
\end{proof}

\section{Virtual fundamental classes}
\label{sec:obs}

The purpose of this section is to introduce relative obstruction theories for the maps
\begin{gather*}
\oM^{\rel}(X,D) \rightarrow \fM^{\rel}(\sA,\sD) \\
\oM^{\orb}(X_r) \rightarrow \fM^{\orb}(\sA_r)
\end{gather*}
and show that these obstruction theories can be used to define the virtual fundamental classes of $\oM^{\rel}(X,D)$ and $\oM^{\orb}(X_r)$.  The reader who is willing to accept that the virtual fundamental classes constructed here agree with the usual ones---or to take the construction given here as the definition---may prefer to read only the constructions of the obstruction theories and proceed to the next section.

\subsection{Obstruction theories}
\label{sec:obs-thy}

We will use the formalism for obstruction theories introduced in~\cite{obs}.  The obstruction theories of loc.\ cit.\ are essentially equivalent to those defined by Behrend and Fantechi~\cite{BF}, but the definition of~\cite{obs} avoids the cotangent complex and therefore makes certain verifications easier.  We briefly recall the definition.

Let $p : X \rightarrow Y$ be a morphism of Deligne--Mumford type.  By a \emph{square-zero lifting problem} for $p$ we will mean a diagram~\eqref{eqn:8} in which $S'$ is a square-zero extension of $S$ with ideal sheaf $J$:
\begin{equation} \label{eqn:8} \vcenter{\xymatrix{
S \ar[r] \ar[d] & X \ar[d]^p \\
S' \ar@{-->}[ur] \ar[r] & Y 
}} \end{equation}
Solutions to the lifting problem are dashed arrows rendering the whole diagram commutative.

An \emph{obstruction theory} $\mathfrak{E}$ for $p$ associates to any $X$-scheme $S$ and any quasi-coherent sheaf $J$ on $S$ an \emph{obstruction groupoid} $\mathfrak{E}(S,J)$ and to any square-zero lifting problem~\eqref{eqn:8} an obstruction $\omega \in \mathfrak{E}(S,J)$.  The obstructions and obstruction groups are required to satisfy various compatibility conditions that we summarize briefly:
\begin{enumerate}[label=(\roman{*})]
\item (functoriality) $\mathfrak{E}(S,J)$ is contravariant in $S$, covariant in $J$, and covariant with respect to \emph{affine morphisms} in $S$;
\item (descent) $\mathfrak{E}(S,J)$ is a stack in the \'etale topology on $S$;
\item (naturality of obstructions) the obstruction $\omega$ is natural in $S$ with respect to \'etale pullback, natural in $S$ with respect to affine pushout, and natural in $J$ with respect to pushout of extensions, see \cite[Definition~3.2]{obs};
\item (additivity) the morphism $\mathfrak{E}(S,J \times J') \rightarrow \mathfrak{E}(S,J) \times \mathfrak{E}(S,J')$ is an equivalece for any quasi-coherent sheaves $J$ and $J'$ on $S$;
\item \label{item:obs-leftexact} (left exactness) if
\begin{equation*}
0 \rightarrow J' \rightarrow J \rightarrow J'' \rightarrow 0
\end{equation*}
is an exact sequence of quasi-coherent sheaves on $S$ then the sequence
\begin{equation} \label{Eq:obs-leftexact} 
0 \rightarrow \mathfrak{E}(S,J') \rightarrow \mathfrak{E}(S,J) \rightarrow \mathfrak{E}(S,J'')
\end{equation}
is also exact.
\end{enumerate}

\begin{remark}
Condition \ref{item:obs-leftexact}  above is best motivated via {\em homogeneity} of deformations over trivial square-0 extensions. Since $J' = J\times_{J''} 0$, we have $\cO_S[J'] = \cO_S[J]\times_{\cO_S[J'']} \cO_S$. Schlessinger's axioms in their strong form require that $\mathfrak{E}(S,J') = \mathfrak{E}(S,J) \times_{\mathfrak{E}(S,J'')} S$, which is precisely Equation \ref{Eq:obs-leftexact}.

In fact an obstruction theory for $X$ over $Y$ may be viewed as the necessary collection of data to extend the definition of $X$ to a moduli problem over $Y$ defined on a small class of derived schemes.  The conditions above combine to require this extension be homogeneous.
\end{remark}

The standard way of producing an obstruction theory for $X$ over $Y$ is to identify some refinement $\widetilde{S}$ of the topology of $S$ over which the lifting problem~\eqref{eqn:8} becomes locally trivial but still satisfies descent.  By abstract nonsense, lifts of problem~\eqref{eqn:8} over $\widetilde{S}$ form a gerbe banded by a sheaf of abelian groups $T$ that only depends on $J$.  Then one may take $\mathfrak{E}(S,J) = H^2(\widetilde{S}, T)$ and the class of the aforementioned gerbe in $H^2(\widetilde{S},T)$ is the obstruction.

\subsubsection{Virtual fundamental classes}

Associated to an obstruction theory is an $\mathcal{O}_X$-module stack $\mathfrak{E}_{X/Y}$ whose value on an $X$-scheme $S$ is
\begin{equation*}
\mathfrak{E}_{X/Y}(S) = \mathfrak{E}(S,\mathcal{O}_S) .
\end{equation*}
By~\cite{BF} or~\cite{obs} there is a canonical embedding of the relative intrinsic normal cone stack $\mathfrak{C}_{X/Y}$ in $\mathfrak{E}_{X/Y}$, for any obstruction theory $\mathfrak{E}$.  Should $\mathfrak{E}_{X/Y}$ be a \emph{vector bundle stack} and $Y$ be pure dimensional, one may apply \cite[Proposition~4.3.2]{Kresch} and obtain a virtual fundamental class $[X/Y]^{\rm vir}$ by intersecting $\mathfrak{C}_{X/Y}$ with the zero locus in $\mathfrak{E}_{X/Y}$.

Manolache observed that this construction applies to any cycle in $A_\ast(Y)$ and therefore defines a Gysin pullback homomorphism on Chow groups~\cite{Manolache}:
\begin{equation*}
p^! : A_\ast(Y) \rightarrow A_\ast(X) 
\end{equation*}
One recovers $[X/Y]^{\rm vir}$ as $p^! [Y]$.

\subsubsection{Compatibility of obstruction theories}
\label{sec:Manolache}

Suppose that $X \xrightarrow{f} Y \xrightarrow{g} Z$ is a sequence of morphisms of Deligne--Mumford type and that $\mathfrak{E}$ and $\mathfrak{F}$ are relative obstruction theories for $X$ over $Z$ and for $Y$ over $Z$, respectively.  Assume that, for every $X$-scheme $S$ and every quasi-coherent sheaf $J$ on $S$ we have a morphism
\begin{equation} 
\Phi : \mathfrak{E}(S,J) \rightarrow \mathfrak{F}(S,J)
\end{equation}
that is compatible with the functoriality properties of $\mathfrak{E}$ and $\mathfrak{F}$ and the naturality properties of the obstructions.  Let $\mathfrak{G}(S,J)$ be the kernel of $\Phi$.  Then $\mathfrak{G}(S,J)$ is a relative obstruction theory for $X$ over $Y$.

If, moreover, the maps $\Phi$ are surjective (as morphisms of stacks) then we say the sequence~\eqref{eqn:9} is exact.
\begin{equation} \label{eqn:9}
0 \rightarrow \mathfrak{G} \rightarrow \mathfrak{E} \rightarrow \mathfrak{F} \rightarrow 0
\end{equation}
When $\mathfrak{E}$, $\mathfrak{F}$, and $\mathfrak{G}$ form an exact sequence of perfect relative obstruction theories, Manolache showed that their associated virtual fundamental classes are compatible:  One has $f^! g^! = (gf)^!$ and therefore
\begin{equation*}
f^! [Y/Z]^{\rm vir} = f^! g^! [Z] = (gf)^! [Z] = [X/Z]^{\rm vir} .
\end{equation*}
In particular, if one can arrange for $[Y/Z]^{\rm vir}$ to coincide with $[Y]$ (i.e., if $Y$ is a local complete intersection relative to $Z$; see Appendix~\ref{sect:lci}) then one has
\begin{equation*}
[X/Y]^{\rm vir} = f^! [Y] = f^! [Y/Z]^{\rm vir} = [X/Z]^{\rm vir} .
\end{equation*}

\subsection{Orbifold maps}
\label{sec:obs-orb}

Recal that, by convention the virtual fundamental class of $\oM^{\orb}(X)$ is defined to be the class associated to the perfect relative obstruction theory for $\fM^{\orb}(X)$ relative to $\fM^{\orb}$~\cite[Section~4.5]{AGV-GW}.  We recall the construction of this obstruction theory below and show it yields the same virtual fundamental class as other obstruction theories that are more convenient for our use.

Let $X \rightarrow Y$ be a smooth Deligne--Mumford-type morphism of algebraic  stacks.  There is an induced projection $\fM^{\orb}(X) \rightarrow \fM^{\orb}(Y)$ by composition with the map $X \rightarrow Y$.  We may construct a relative obstruction theory for this projection by considering an infinitesimal lifting problem
\begin{equation} \label{eqn:31} \vcenter{\xymatrix{
S \ar[r] \ar[d] & \fM^{\orb}(X) \ar[d] \\
S' \ar[r] \ar@{-->}[ur] & \fM^{\orb}(Y) 
}} \end{equation}
in which $S'$ is a square-zero extension of $S$ with ideal $J$.  This corresponds to an extension problem for maps:
\begin{equation*} \xymatrix{
& & X \ar[d] \\
C \ar[r] \ar@/^10pt/[urr]^f \ar[d]_\pi & C' \ar@{-->}[ur] \ar[d] \ar[r] & Y \\
S \ar[r] & S'
} \end{equation*}
This simplifies to the following lifting problem:
\begin{equation*} \xymatrix{
C \ar[r]^f \ar[d] & X \ar[d] \\
C' \ar[r] \ar@{-->}[ur] & Y
} \end{equation*}
whose lifts form a torsor on $C$ under the sheaf of abelian groups $f^\ast T_{X/Y} \tensor \pi^\ast J$.  Defining $\fE(S, J)$ to be the category of torsors under $f^\ast T_{X/Y} \tensor \pi^\ast J$ we therefore obtain a section of $\fE(S,J)$ obstructing the existence of a dashed arrow completing diagram~\ref{eqn:31}.  Thus $\fE$ forms a relative obstruction theory for the map $\fM^{\orb}(X) \rightarrow \fM^{\orb}(Y)$.

If we have a sequence of morphisms $X \rightarrow Y \xrightarrow{g} Z$ of Deligne--Mumford type, we obtain a sequence of maps 
\begin{equation*}
\fM^{\orb}(X) \rightarrow \fM^{\orb}(Y) \xrightarrow{h} \fM^{\orb}(Z) .
\end{equation*}
The exactness of the sequence
\begin{equation*}
0 \rightarrow T_{X/Y} \rightarrow T_{X/Z} \rightarrow g^\ast T_{Y/Z} \rightarrow 0
\end{equation*}
gives rise to a sequence of compatible obstruction theories
\begin{equation*}
0 \rightarrow \fE_{\fM^{\orb}(X) / \fM^{\orb}(Y)} \rightarrow \fE_{\fM^{\orb}(X) / \fM^{\orb}(Z)} \rightarrow h^\ast \fE_{\fM^{\orb}(Y) / \fM^{\orb}(Z)} \rightarrow 0 .
\end{equation*}
Exactness on the left is formal and exactness on the right follows from the vanishing of $H^2(C, f^\ast g^\ast T_{Y/Z})$ at a point $f : C \rightarrow X$ of $\fM^{\orb}(X)$ (because $f^\ast g^\ast T_{Y/Z}$ is quasi-coherent and $C$ is a curve).

When $X$ is a proper Deligne--Mumford stack, the virtual fundamental class for $\oM^{\orb}(X)$ is constructed using the obstruction theory defined as above for the morphism from $X$ to a point. 

We will apply this in the case where $X$ is equipped with a smooth divisor $D \subset X$ giving rise to a morphism $X \rightarrow \sA$.  We therefore have a sequence of morphisms of Deligne--Mumford type
\begin{equation*}
X \rightarrow \sA \rightarrow \BGm \rightarrow (\text{point})
\end{equation*}
giving rise to
\begin{equation*}
\fM^{\orb}(X) \rightarrow \fM^{\orb}(\sA) \rightarrow \fM^{\orb}(\BGm) \rightarrow \fM^{\orb} .
\end{equation*}
Since $\fM^{\orb}(\BGm)$ is smooth and unobstructed, the compatiblity of the obstruction theories in the sequence
\begin{equation*}
\fM^{\orb}(X) \rightarrow \fM^{\orb}(\BGm) \rightarrow \fM^{\orb}
\end{equation*}
shows that the virtual fundamental class of $\oM^{\orb}(X)$ can be defined relative to $\fM^{\orb}(\BGm)$.  On the other hand, the morphisms in the sequence
\begin{equation*}
\fM^{\orb}(X) \rightarrow \fM^{\orb}(\sA) \rightarrow \fM^{\orb}(\BGm)
\end{equation*}
are of Deligne--Mumford type and have compatible obstruction theories.  Therefore the virtual fundamental class of $\fM^{\orb}(X)$ is the virtual pullback, via the relative obstruction theory of $\fM^{\orb}(X) / \fM^{\orb}(\sA)$ of the relative virtual fundamental class of $\fM^{\orb}(\sA) / \fM^{\orb}(\BGm)$.

Now we restrict attention to genus~$0$ maps and place ourselves in the case where $X$ is replaced by $X_r = \sqrt[r]{X,D}$, for some smooth scheme $X$.  We select $r$ accoring to the following criteria:

\begin{lemma} \label{lem:twisting-choice}
Let $\beta$ be an effective class in $H_2(X, \ZZ)$ and let $d = D . \beta$.  Set $\kappa = \max_{0 \leq \gamma \leq \beta} \bigl| D . \gamma \bigr|$, the maximum taken over all classes $\gamma$ such that both $\gamma$ and $\beta - \gamma$ are effective.  Let $r$ be an integer larger than $2 \kappa$ and all of the contact orders $k_i$.  Then the map $\oM^{\orb}_{g=0}(X_r,\beta) \rightarrow \fM^{\orb}_{g=0}(\sA_r)$ factors through in $\fM^{\orb}_{g=0}(\sA_r)'$ (Section~\ref{sec:orb-dense-open}).
\end{lemma}
\begin{proof}
Consider a map $f : C \rightarrow X_r$ with $f_\ast [C] = \beta$.  If $C_0 \subset C$ is a proper subcurve and $\gamma = f_\ast [C_0]$ then, we have
\begin{equation*}
\bigl| D_r . \gamma \bigr| = \Bigl| \frac{1}{r} D . \gamma \Bigr| \leq \Bigl| \frac{\kappa}{r} \Bigr| < \frac{1}{2} .
\end{equation*}
This gives Condition~\ref{cond:2} of Section~\ref{sec:orb-dense-open}.  To get the Condition~\ref{cond:1} of Section~\ref{sec:orb-dense-open}, recall that $D . \beta = \sum k_i$ by assumption (see Section~\ref{sec:thm-statement}).  Thus $D_r . \beta = \sum \frac{k_i}{r}$.  But $0 \leq \frac{k_i}{r} < 1$ so $\frac{k_i}{r}$ is precisely the age of $f^\ast \mathcal{O}_{X_r}(D_r)$ at the $i$-th marked point. 
\end{proof}

According to the lemma, our choice of $r$ guarantees that the map $\oM^{\orb}_{g=0}(X_r, \beta) \rightarrow \fM^{\orb}(\sA)$  factors through $\fM_{g=0}^{\orb}(\sA)'$.  By Lemma~\ref{lem:non-degen-dense-in-M-orb}, the stack $\fM^{\orb}_{g=0}(\sA)'$ contains a dense open substack that is unobstructed relative to $\fM^{\orb}(\BGm)$.  It follows by Lemma~\ref{lem:obs-lci} that $\fM^{\orb}_{g=0}(\sA)' \rightarrow \fM^{\orb}(\BGm)$ is a local complete intersection morphism and that the virtual fundamental class of $\fM^{\orb}_{g=0}(\sA)'$ defined using this relative obstruction theory coincides with the fundamental class.  By Manolache's theorem \cite[Theorem~4.8]{Manolache} it now follows that the virtual fundamental class of $\oM^{\orb}_{g=0}(X_r,\beta)$ may be defined relative to $\fM^{\orb}_{g=0}(\sA)'$:

\begin{proposition}
Assume that $r$ satisfies
the conditions of Lemma~\ref{lem:twisting-choice}.
Then the virtual fundamental class of $\oM^{\orb}_{g=0}(X_r, \beta)$ relative to $\fM$ coincides with the virtual fundamental class relative to $\fM^{\orb}_{g=0}(\sA)'$.
\end{proposition}

\subsection{Relative maps}
\label{sec:obs-rel}

We consider the sequence of maps
\begin{equation*}
\fM^{\rel}(X,D) \rightarrow \fM^{\rel}(\sA, \sD) \rightarrow \fM \rightarrow (\text{point}) 
\end{equation*}
and define compatible relative obstruction theories for some of the maps in the sequence.  We conclude that all define the same virtual fundamental class on $\oM^{\rel}(X,D)$.

\subsubsection{J.\ Li's obstruction theory}

We first describe Li's absolute obstruction theory for $\fM^{\rel}(X,D)$.

Let $C \rightarrow S$ be a family of nodal curves.  Define $\et(C/S)$ to be the category of commutative diagrams
\begin{equation*} \xymatrix{
U \ar[r] \ar[d] & C \ar[d] \\
V \ar[r] & S 
} \end{equation*}
where the horizontal arrows are \'etale.  Such an object is abbreviated $UV$.  We give this category the topology where a family of maps $U' V' \rightarrow UV$ is covering if the families of maps $U' \rightarrow U$ and $V' \rightarrow V$ are covering in the \'etale topology.  See \cite[Section~4.2]{CMW} or \cite[Section~3.2.3]{AMW} for more about this topology.

Let $(X, D)$ be a smooth pair and $\tX$ an expansion of $(X,D)$ parameterized by $S$.  Let
\begin{equation*} \xymatrix{
C \ar[r] \ar[d] & X^{\exp} \ar[d] \\
S \ar[r] & \sT
} \end{equation*}
be an $S$-point of $\fM^{\rel}(X,D)$.  If $S \subset S'$ is a square-zero extension with ideal $J$ one may ask for extensions:
\begin{equation} \label{eqn:32} \vcenter{\xymatrix{
C \ar@/^15pt/[rr] \ar@{-->}[r] \ar[d]_{\pi} & C' \ar@{-->}[d] \ar@{-->}[r] & X^{\exp} \ar[d] \\
S \ar@{->}[r] \ar@/_15pt/[rr] & S' \ar@{-->}[r] & \sT
}} \end{equation}

Such extensions form a stack on the site $\et(C/S)$.  In fact, if we consider the special case where $S' = S[J]$ and $C' = C[\pi^\ast J]$ are the trivial square-zero extensions, then an extension is guaranteed to exist and the collection of all extensions forms a stack of commutative $2$-groups on $\et(C/S)$.  We denote this sheaf $T(C/S, J)$.  In~\cite[Lemma~1.12]{Li2}, Li shows that when the extensions $S'$ and $C'$ are non-trivial, the solutions to the extension problem~\eqref{eqn:32} exist locally in $\et(C/S)$.  It then follows formally that solutions form a torsor on $\et(C/S)$ under $T(C/S, J)$.%
\footnote{Loc.\ cit.\ gives a different calculation of the structure group of this torsor that appears to be correct only over objects $UV$ of $\et(C/S)$ where $U$ covers $V$. This issue is clarified in \cite[Sections 4.3 and A.1]{CMW}.}

Writing $\fE(S, J)$ for the category of $T(C/S,J)$-torsors on $\et(C/S)$, it follows that there is a section $\omega \in \fE(S, J)$ obstructing the existence of a solution to~\eqref{eqn:32}.  That is, $\fE$ is an obstruction theory for $\fM^{\rel}(X,D)$.

We note that in the absense of a stability assumption, $\fE(S,J)$ will be a $2$-category.  However, at an $S$-point of $\oM^{\rel}(X,D)$ the category $\fE(S,J)$ of obstructions will be a $1$-category.  Indeed, the $2$-automorphisms in $\fE(S,J)$ are infinitesimal automorphisms of the moduli problem.

\subsubsection{J.\ Li's relative obstruction theory}

Li's description of the obstruction from the last section is more explicit and goes by way of a relative obstruction theory for $\fM^{\rel}(X,D)$ over $\fM$.  To study this obstruction we consider the lifting problem
\begin{equation*}  \xymatrix{
S \ar[r] \ar[d] & \fM^{\rel}(X,D) \ar[d] \\
S' \ar[r] \ar@{-->}[ur] & \fM
} \end{equation*}
where $\fM$ is the stack of pre-stable curves and $S'$ is a square-zero extension of $S$ with ideal $J$.  This problem translates into the following one:
\begin{equation*} \xymatrix{
C \ar[r] \ar[d]_{\pi} \ar@/^15pt/[rr] & C' \ar@{-->}[r] \ar[d] & X^{\exp} \ar[d] \\
S \ar[r] \ar@/_15pt/[rr] & S' \ar@{-->}[r] & \sT .
}\\[2mm] \end{equation*}

Li shows in \cite[Lemma~1.12]{Li2} that this problem also admits solutions locally in $\et(C/S)$.  It follows that the solutions form a torsor under the stack of commutative $2$-groups $T'(C/S, J)$ defined to be the collection of solutions to the above problem with $S' = S[J]$ and $C' = C[\pi^\ast J]$.  The collection of all such torsors, denoted $\fE'(S,J)$, therefore forms an obstruction theory for $\fM^{\rel}(X,D)$ over $\fM$.

We will write $T''(C/S,J)$ for the stack of commutative $2$-groups on $\et(C/S)$ whose sections over $UV$ are extensions
\begin{equation*} \xymatrix{
U \ar@{-->}[r] \ar[d] & U' \ar@{-->}[d] \\
V \ar[r] & V[J \rest{V}]
} \end{equation*}
We note that $T''(C/S,J)(UV)$ depends only on $U$ so $T''(C/S,J)$ is pushed forward via the closed embedding $\et(C) \rightarrow \et(C/S)$.  As pushforward via a closed embedding is exact, the cohomology of $T''(C/S,J)$ agrees with the cohomology of the corresponding sheaf on $\et(C)$, which one can calculate to be $\uHom(\bL_{C/S}[-1], \pi^\ast J)$ on $\et(C)$.

Now we have an exact sequence
\begin{equation*}
0 \rightarrow T'(C/S,J) \rightarrow T(C/S,J) \rightarrow T''(C/S,J) \rightarrow 0 .
\end{equation*}
As 
\begin{equation*}
H^1(\et(C/S), T''(C/S,J)) = \Ext^2(\bL_{C/S}, \pi^\ast J) = 0,
\end{equation*}
we get an exact sequence
\begin{equation*}
0 \rightarrow \sT''(S,J) \rightarrow \fE'(S,J) \rightarrow \fE(S,J) \rightarrow 0
\end{equation*}
and therefore by
\cite[Proposition~6.5]{obs}, the obstruction theories
$\fE'$ and $\fE$ define the same virtual fundamental class on $\oM^{\rel}(X,D)$.

\subsubsection{The obstruction theory relative to the universal moduli space}
\label{sec:univ-obs}

In \cite{AMW}, it was shown that a similar construction to the above gives a relative obstruction theory for the map $\fM^{\rel}(X,D) \rightarrow \fM^{\rel}(\sA, \sD)$.  Moreover this obstruction theory agrees with the following one:  Consider the lifting problem
\begin{equation}  \label{eqn:36} \vcenter{\xymatrix{
S \ar[r] \ar[d] & \fM^{\rel}(X,D) \ar[d] \\
S' \ar[r] \ar@{-->}[ur] & \fM^{\rel}(\sA,\sD) 
}} \end{equation}
corresponding to the extension problem
\begin{equation*} \xymatrix{
& & X^{\exp} \ar[d] \\
C \ar[r] \ar[d]_\pi \ar@/^10pt/[urr] & C' \ar@{-->}[ur] \ar[r] \ar[d] & \sA^{\exp} \ar[d]  \\
S \ar[r] & S' \ar[r] & \sT ,
} \end{equation*}
which reduces immediately to
\begin{equation} \label{eqn:33} \vcenter{\xymatrix{
C \ar[r] \ar[d] & X^{\exp} \ar[d] \\
C' \ar[r] \ar@{-->}[ur] & \sA^{\exp} .
}} \end{equation}
Noting we have a cartesian diagram
\begin{equation*} \xymatrix{
X^{\exp} \ar[r] \ar[d] & X \ar[d] \\
\sA^{\exp} \ar[r] & \sA
} \end{equation*}
it follows that lifts of~\eqref{eqn:33} form a torsor under $f^\ast T_{X/\sA} \tensor \pi^\ast J = f^\ast T_X^{\log} \tensor \pi^\ast J$.  Here $f$ denotes the composition $C \rightarrow X^{\exp} \rightarrow X$.  It follows that we have a perfect relative obstruction theory $\fE''$ with $\fE''(S,J)$ being the category of torsors on $C$ under $f^\ast T_X^{\log} \tensor \pi^\ast J$.

We now argue that this obstruction theory yields the same virtual fundamental class as the relative obstruction theory $\fE'$ for $\fM^{\rel}(X,D)$ over $\fM$.  Let $T_{\fM^{\rel}(X,D)/\fM^{\rel}(\sA,\sD)}$ be the pushforward of the sheaf $f^\ast T_X^{\log} \tensor \pi^\ast J$ along the closed embedding $\et(C)$ to $\et(C/S)$.  The torsors under $T_{\fM^{\rel}(X,D)/\fM^{\rel}(\sA,\sD)}$ are the same as the torsors under $f^\ast T_X^{\log} \tensor \pi^\ast J$.  There is an exact sequence
\begin{equation*}
0 \rightarrow T_{\fM^{\rel}(X,D)/\fM^{\rel}(\sA,\sD)} \rightarrow T_{\fM^{\rel}(X,D)/\fM} \rightarrow T_{\fM^{\rel}(\sA,\sD)/\fM} \rightarrow 0
\end{equation*}
on $\et(C/S)$ yielding an exact sequence of obstruction theories
\begin{equation*}
0 \rightarrow \fE_{\fM^{\rel}(X,D)/\fM^{\rel}(\sA,\sD)} \rightarrow \fE_{\fM^{\rel}(X,D)/\fM} \rightarrow \fE_{\fM^{\rel}(\sA/\sD)/\fM} \rightarrow 0 .
\end{equation*}
Exactness on the right comes from the vanishing of $H^2(C,f^\ast T_X^{\log} \tensor \pi^\ast J)$ as $f^\ast T_X^{\log} \tensor \pi^\ast J$ is quasicoherent and $C$ is a curve.

\begin{lemma}
The map $\fM^{\rel}(\sA,\sD) \rightarrow \fM$ is representable by algebraic spaces.
\end{lemma}
\begin{proof}
Consider\Jonathan{Should I rewrite this?} a geometric point of $\fM^{\rel}(\sA,\sD)$ 
\begin{equation*}
f : C \rightarrow \tsA
\end{equation*}
where $\tsA$ is an expansion of $\sA$.  We wish to show there are no infinitesimal automorphisms of this object fixing $C$.  An automorphism fixing $C$ of such an object is a commutative diagram
\begin{equation*} \xymatrix{
C \ar[r]^{f} \ar[dr]_f & \tsA \ar[d] \\
& \tsA
} \end{equation*}
Recall that $\tsA$ is a chain of copies of $\sP = [ \mathbf{P}^1 / \Gm ]$, joined at nodes.  There is an open neighborhood of each node $\sE \subset \tsA$ where two copies of $\sA$---call them $\sA^+$ and $\sA^-$ and write $\sE^+ \simeq \BGm$ and $\sE^- \simeq \BGm$ for their distinguished divisors---are joined.  Moreover, these copies are joined by isomorphisms $\sE^- \simeq \sE^+$ so that $N_{\sE^-/\sA^-} \tensor N_{\sE^+/\sA^+} \simeq L$ where $L$ is a \emph{specified} line bundle pulled back from the base.  The automorphism group of $\tsA$ is a canonically split torus of rank equal to the number of nodes, with the factor corresponding to $\sE$ acting by scaling $\sL$.

If we fix a node $\sE \subset \tsA$ there is some node $x \in C$ that is carried by $f$ to $\sE$.  Letting $C^-$ and $C^+$ be the two components of $C$ joined at $x$, we get
\begin{gather*}
f^\ast N_{\sE^-/\sA^-} = N_{x^-/C^-}^{\tensor r} \\
f^\ast N_{\sE^+/\sA^+} = N_{x^+/C^+}^{\tensor r}
\end{gather*}
where $r$ is the order of contact of $f$ to $\sE$ at $x$.  Note that there is a canonical identification $N_{x^-/C^-} \tensor N_{x^+/C^+}$ with the deformation space of the node $x$, which is a line bundle on the base.  By predeformability, the identification
\begin{equation*}
f^\ast N_{\sE^-/\sA^-} \tensor f^\ast N_{\sE^+/\sA^+} \simeq \cO
\end{equation*}
is the $r$-th power of the identification of $N_{x^-/C^-} \tensor N_{x^+/C^+}$ with the deformation space of the node.  In particular, the scaling of $L$ is forced by $f$ to be the identity.  Thus there are no nontrivial automorphisms of $\tsA$ commuting with $f$.  That is, $\fM^{\rel}(\sA,\sE) \rightarrow \fM$ is representable by algebraic spaces.
\end{proof}

By Lemma~\ref{lem:rel-dense-open}, there is a dense open subset of $\fM^{\rel}(\sA,\sD)$ that is unobstructed relative to $\fM$.  Thus $\fM^{\rel}(\sA,\sD)$ is a local complete intersection (Lemma~\ref{lem:obs-lci}) with its canonical obstruction theory and the relative virtual fundamental class over $\fM$ is simply the fundamental class.  It follows that
\begin{equation*}
\bigl[ \fM^{\rel}(X,D) \big/ \fM^{\rel}(\sA,\sD) \bigr]^{\vir} = \bigl[ \fM^{\rel}(X,D) \big/ \fM \bigr]^{\rm vir}
\end{equation*}
as required.

\section{Proof of Theorem~\ref{Prop1}}
\label{sec:proof1}

The map 
\begin{equation*}
\Psi = \Psi^{(X,D)} : \oM^\rel(X_r, D_r) \rightarrow \oM^\rel(X,D)
\end{equation*}
fits into a cartesian diagram
\begin{equation} \label{eqn:29} \vcenter{\xymatrix@C=45pt{
\oM^\rel(X_r, D_r) \ar[r]^{\Psi^{(X,D)}} \ar[d] & \oM^\rel(X,D) \ar[d]  \\
\fM^\rel(\sA_r, \sD_r) \ar[r]^{\Psi^{(\sA,\sD)}} \ar[r] & \fM^\rel(\sA,\sD)
}} \end{equation}
in which $(\sA, \sD)$ is the \textit{universal smooth pair}:  $\sA = [ \bA^1 / \Gm]$ and $\sD = [ \: 0 \: / \Gm ] \subset \sA$.  In fact, $(\sA_r, \sD_r) \simeq (\sA, \sD)$ but we use the subscript to emphasize that the map $(\sA_r, \sD_r) \rightarrow (\sA, \sD)$ inducing $\Psi^{(\sA,\sD)}$ is not the identity.

Section~\ref{sec:obs} showed that the virtual fundamental classes of $\oM^\rel(X,D)$ and $\oM^\rel(X_r, D_r)$ are defined, respectively, relative to $\fM^\rel(\sA,\sD)$ and $\fM^\rel(\sA_r, \sD_r)$.  That is, there are perfect relative obstruction theories $\fE$ for $\oM^{\rel}(X,D)$ over $\fM^{\rel}(\sA, \sD)$ and $\fF$ for $\oM^{\rel}(X_r, D_r)$ over $\fM^{\rel}(\sA_r, \sD_r)$ such that the virtual fundamental classes of $\oM^{\rel}(X,D)$ and $\oM^{\rel}(X_r, D_r)$ are pulled back via these obstruction theories from the fundamental classes of $\fM^{\rel}(\sA,\sD)$ and $\fM^{\rel}(\sA_r, \sD_r)$.

In order to apply Costello's theorem~\cite[Theorem~5.0.1]{costello}, we must show that $\fF$ is pulled back via $\Psi^{(X,D)}$ from $\fE$.  We recall the definitions of $\fE$ and $\fF$:  given an extension problem
\begin{equation*} \xymatrix{
S \ar[r] \ar[d] & \oM^{\rel}(X, D) \ar[d] \\
S' \ar[r] \ar@{-->}[ur] & \fM^{\rel}(\sA, \sD)
} \end{equation*}
corresponding to an extension problem
\begin{equation*} \xymatrix{
& & \tX \ar[d]  \\
C \ar[r] \ar[d]_{\pi} \ar@/^10pt/[urr]^f & C' \ar[r] \ar@{-->}[ur] \ar[d] & \tsA \ar[dl] \\
S \ar[r] & S'
} \end{equation*}
the solutions to this problem form a torsor under $f^\ast T_{\tX / \tsA} \tensor \pi^\ast J = g^\ast T_X(- \log D) \tensor \pi^\ast J$ where $g$ denotes the composition $C \rightarrow \tX \rightarrow X$.  The obstruction theory $\fE(S,J)$ is the category of  $(f^\ast T_{\tX / \tsA} \tensor \pi^\ast J)$-torsors on $C$ and the obstruction is the class of this torsor.

The obstruction theory for $\oM^{\rel}(\sX_r, \sD_r)$ over $\fM^{\rel}(\sA_r, \sD_r)$ is defined the same way.  But now we note that there is a cartesian diagram
\begin{equation*} \xymatrix{
\tX_r \ar[r] \ar[d] & \tX \ar[d] \\
\tsA_r \ar[r] & \tsA
} \end{equation*}
so that the torsor of lifts of a diagram
\begin{equation*} \xymatrix{
C \ar[r] \ar[d] & \tX_r \ar[d] \\
C' \ar[r] \ar@{-->}[ur] & \tsA_r
} \end{equation*}
is precisely the same as the torsor of lifts of the induced diagram
\begin{equation*} \xymatrix{
C \ar[r] \ar[d] & \tX \ar[d] \\
C' \ar[r] \ar@{-->}[ur] & \tsA .
} \end{equation*}
This is precisely what it means for $\fF$ to be pulled back from $\fE$.

The only remaining requirement of Costello's theorem is to show that $\fM^{\rel}(\sA_r, \sD_r) \rightarrow \fM^{\rel}(\sA, \sD)$ has pure degree~$1$.  This is Theorem~\ref{thm:rel-birational}.

\section{Proof of Theorem~\ref{Prop2}}
\label{sec:proof2}

\subsection{The cartesian diagram}

This proof follows the same lines as the proof of Theorem~\ref{Prop1}.  The key difference is that the diagram one would naively expect to replace~\eqref{eqn:29} does not commute!  All is not lost, however:  a small modification restores the commutativity.

Define $\fM(\sA_r, \sD_r)^\ast$ to be the moduli space whose $S$-points are commutative diagrams of $S$-stacks
\begin{equation} \label{eqn:34} \vcenter{\xymatrix{
C \ar[r] \ar[d] & \tsA_r \ar[d] \\
\oC \ar[r] & \sA_r \times S
}} \end{equation}
where 
\begin{enumerate}
\item $C$ and $\oC$ are orbifold curves over $S$, 
\item the maps $C \rightarrow \tsA_r$ and $\oC \rightarrow \sA_r$ are twisted pre-stable maps over $S$,
\item the map $C \rightarrow \oC$ is representable and its stabilization is an isomorphism,
\item $\tsA_r$ is an expansion of $\sA_r$ over $S$,
\item the automorphism group of these data fixing $\oC \rightarrow \sA_r$ is finite.
\end{enumerate}

Then we have a diagram
\begin{equation} \label{eqn:30} \vcenter{\xymatrix{
\oM^\rel(X_r, D_r) \ar[r] \ar[d] & \oM^\orb(X_r) \ar[d] \\
\fM^\rel(\sA_r, \sD_r)^\ast \ar[r] & \fM^\orb(\sA_r) .
}} \end{equation}
The upper horizontal map composes a stable relative map $C \rightarrow \tX_r$ with the projection $\tX_r \rightarrow X_r$ and stabilizes the result.  The vertical arrow on the left sends $C \rightarrow \tX_r$ to the outer rectangle in the commutative diagram
\begin{equation*} \xymatrix{
C \ar[r] \ar[d] & \tX_r \ar[r] \ar[d] & X_r \ar[d] \\
\oC \ar[r] & \tsA_r \ar[r] & \sA_r
} \end{equation*}
where $\oC$ is the stabilization of the composition $C \rightarrow \tX_r \rightarrow X_r$.  The lower horizontal arrow in diagram~\eqref{eqn:30} sends a diagram~\eqref{eqn:34} to its bottom half.  Finally the right vertical arrow in~\eqref{eqn:30} sends $C \rightarrow X_r$ to the composition with $X_r \rightarrow \sA_r$.

\begin{proposition}
Diagram~\eqref{eqn:30} is commutative and cartesian.
\end{proposition}
\begin{proof}
It is clear that the diagram is commutative.  It is also immediate from the cartesian diagram
\begin{equation*} \xymatrix{
X_r^{\exp} \ar[r] \ar[d] & X_r \ar[d] \\
\sA_r^{\exp} \ar[r] & \sA_r
} \end{equation*}
that the diagram
\begin{equation*} \xymatrix{
\fM^{\rel}(X_r, D_r)^\ast \ar[r] \ar[d] & \fM^{\orb}(X_r) \ar[d] \\
\fM^{\rel}(\sA_r, \sD_r)^\ast \ar[r] & \fM^{\orb}(\sA_r)
} \end{equation*}
is cartesian.  It remains only to show that the map $\oM^{\rel}(X_r, D_r) \rightarrow \fM^{\rel}(X_r, D_r)^\ast$ identifies $\oM^{\rel}(X_r, D_r)$ with the pre-image of the open substack $\oM^{\orb}(X_r) \subset \fM^{\orb}(X_r)$.  In other words, we are left to show that \emph{the diagram below is cartesian}:
\begin{equation*} \xymatrix{
\oM^{\rel}(X_r, D_r) \ar[r] \ar[d] & \oM^{\orb}(X_r) \ar[d] \\
\fM^{\rel}(X_r, D_r)^\ast \ar[r] & \fM^{\orb}(X_r)
} \end{equation*}

Suppose that $\xi$ denotes an $S$-point
\begin{equation} \label{eqn:35} \vcenter{\xymatrix{
C \ar[r] \ar[d] & \tX_r \ar[d] \\
\oC \ar[r] & X_r \times S
}} \end{equation}
of $\fM^{\rel}(X_r, D_r)^\ast$ that lies in the pre-image of $\oM^{\orb}(X_r)$.  Then $\oC \rightarrow X_r$ can be recovered uniquely and functorially from $C \rightarrow \tX_r$ as the stabilization of the composition $C \rightarrow \tX_r \rightarrow X_r$.  Note in particular that this permits us to identify the automorphism group $\Aut_{\fM^{\rel}(X_r, D_r)^\ast}(\xi)$ with the automorphism group of the pre-stable relative map $C \rightarrow \tX_r$.

Now we verify that $C \rightarrow \tX_r$ has finite automorphism group.  Let $\eta$ denote the image of $\xi$ in $\fM^{\orb}(X_r)$ and consider the exact sequence
\begin{equation*}
1 \rightarrow \Aut_{\fM^{\rel}(X_r, D_r)^\ast \mbox{\larger[-1]/} \fM^{\orb}(X_r)}(\xi) \rightarrow \Aut_{\fM^{\rel}(X_r, D_r)^\ast}(\xi) \rightarrow \Aut_{\fM^{\orb}(X_r)}(\eta) .
\end{equation*}
The group $\Aut_{\fM^{\orb}(X_r)}(\eta)$ is finite by hypothesis because $\eta$ lies in $\oM^{\orb}(X_r)$; likewise $\Aut_{\fM^{\rel}(X_r, D_r)^\ast \mbox{\larger[-1]/} \fM^{\orb}(X_r)}(\xi)$ is finite by the definition of $\tfM^{\rel}(X_r, D_r)^\ast$.  It follows that the middle group $\Aut_{\fM^{\rel}(X_r, D_r)^\ast}(\xi)$ is finite.  On the other hand, we have just seen that this group agrees with the automorphism group of $C \rightarrow \tX_r$ as a pre-stable relative map.

We conclude that Diagram~\ref{eqn:35} can be uniquely recovered as the image of the stable relative map $C \rightarrow \tX_r$ and the proof is complete.
\end{proof}

\subsection{Compatibility of obstruction theories}

We require a perfect relative obstruction theory for the map $\oM^{\rel}(X_r, D_r) \rightarrow \fM^{\rel}(\sA_r, \sD_r)^\ast$.  Fortunately the following lemma allows us to make use of results from Section~\ref{sec:obs}.

\begin{lemma} \label{lem:etale-and-birational}
The map $\fM^{\rel}(\sA_r, \sD_r)^\ast \rightarrow \fM^{\rel}(\sA_r, \sD_r)$ is \'etale and birational.
\end{lemma}
\begin{proof}
A proof that it is \'etale can be adapted easily from \cite[Lemma~B~(ii) for $\Upsilon$]{AMW} and is omitted here.

For birationality, we use an argument similar to \cite[Lemma~B~(iv) for $\Upsilon$]{AMW}:  The map restricts to an isomorphism over the totally nondegenerate objects, which are dense in the target.  Because the map is \'etale, they are also dense in the source.
\end{proof}

\begin{corollary}
A relative obstruction theory for $\oM^{\rel}(X_r, D_r)$ over $\fM^{\rel}(\sA_r, \sD_r)$ is also a relative obstruction theory over $\fM^{\rel}(\sA_r, \sD_r)^\ast$ and yields the same virtual fundamental class.
\end{corollary}

We apply this to the relative obstruction theory for $\oM^{\rel}(X_r, D_r)$ over $\fM^{\rel}(\sA_r, \sD_r)$ constructed in Section~\ref{sec:univ-obs}.  In order to apply Costello's theorem, we must verify that this obstruction theory is pulled back from the obstruction theory for $\oM^{\orb}(X_r)$ over $\fM^{\orb}(\sA_r)$ constructed in Section~\ref{sec:obs-orb}.

Consider a lifting problem where $S'$ is a square-zero extension of $S$ with ideal $J$:  
\begin{equation*} \xymatrix{
S \ar[r] \ar[d] & \oM^{\rel}(X_r,D_r) \ar[d] \\
S' \ar[r] \ar@{-->}[ur] & \fM^{\rel}(\sA_r, \sD_r)
} \end{equation*}
The obstruction group for $\oM^{\rel}(X_r, D_r)$ over $\fM^{\rel}(\sA_r, \sD_r)$ is the category of torsors on $C$ under the group of lifts of the left square in the diagram below:
\begin{equation} \label{eqn:37} \vcenter{\xymatrix{
C \ar[r] \ar[d] & X^{\exp}_r \ar[d] \ar[r] \ar[d] & X_r \ar[d] \\
C[\pi^\ast J] \ar[r] \ar@{-->}[ur] \ar@{-->}[urr] & \sA^{\exp}_r \ar[r] & \sA_r 
}} \end{equation}

On the other hand, the obstruction group for the relative obstruction theory pulled back from the map $\oM^{\orb}(X_r) \rightarrow \fM^{\orb}(\sA_r)$ the the category of torsors on $C$ under the group of lifts of the outer rectangle in~\eqref{eqn:37}.  Since the square on the right in~\eqref{eqn:37} is cartesian, it follows that the two obstruction groups are the same.

\subsection{Conclusion}
\label{sec:conclusion}

Fix combinatorial data $\Gamma$ for a stable map to $X$ relative to a divisor $D$.  These data include the genus of the curve, the homology class of its image in $X$, the number of marked points, and the orders of contact of those marked points to the divisor.  

As long as $\Gamma$ is a specification of combinatorial data that agrees with the requirements of the lemma, we have a cartesian diagram
\begin{equation*} \xymatrix{
\oM^{\rel}_\Gamma(X_r, D_r) \ar[r] \ar[d] & \oM^{\orb}_\Gamma(X_r) \ar[d] \\
\fM^{\rel}_{g=0}(\sA_r, \sD_r)' \ar[r] & \fM^{\orb}_{g=0}(\sA_r)',
} \end{equation*}
where $\fM^{\rel}_{g=0}(\sA_r, \sD_r)'$ denotes the pre-image of $\fM^{\orb}_{g=0}(\sA_r)'$ via the map $\fM^{\rel}_{g=0}(\sA_r, \sD_r)^\ast \rightarrow \fM^{\orb}_{g=0}(\sA_r)$.  We note that the lower horizontal map is of Deligne--Mumford type by the definition of $\fM^{\rel}(\sA_r, \sD_r)^\ast$.  Furthermore, Lemmas~\ref{lem:etale-and-birational},~\ref{lem:rel-dense-open}, and~\ref{lem:non-degen-dense-in-M-orb} show that each of $\fM^{\rel}_{g=0}(\sA_r, \sD_r)'$, $\fM^{\rel}_{g=0}(\sA_r, \sD_r)^\ast$, and $\fM^{\orb}_{g=0}(\sA_r)'$  contains the locus of totally non-degenerate objects as a dense open substack, so the lower horizontal arrow is birational as well.  We may therefore apply Costello's theorem and deduce that
\begin{equation*}
\Phi_\ast \bigl[\oM^{\rel}_\Gamma(X_r, D_r) \mbox{\larger[2]/} \fM^{\rel}(\sA_r, \sD_r)'\bigr]^\vir = \bigl[\oM^{\orb}_\Gamma(X_r) \mbox{\larger[2]/} \fM^{\orb}(\sA_r)' \bigr]^\vir .
\end{equation*}
On the other hand, by the discussion in Sections~\ref{sec:obs-orb} and~\ref{sec:obs-rel} we have
\begin{gather*}
\bigl[\oM^{\orb}_\Gamma(X_r) \mbox{\larger[2]/} \fM^{\orb}(\sA_r)' \bigr]^\vir = [\oM^{\orb}_\Gamma(X_r) ]^\vir \\
\bigl[\oM^{\rel}_\Gamma(X_r,D_r) \mbox{\larger[2]/} \fM^{\rel}(\sA_r, \sD_r)'\bigr]^\vir = [\oM^{\rel}_\Gamma(X_r, D_r)]^\vir .
\end{gather*}

\appendix


\section{Obstruction theories and local complete intersections}
\label{sect:lci}

In this section we use the Behrend--Fantechi formalism for obstruction theories~\cite{BF}.

\begin{lemma} \label{lem:obs-lci}
Let $\sM\to\sN$ be a representable, finite type morphism of locally Noetherian algebraic stacks and let $\EE\to \LL_{\sM/\sN}$ be a perfect relative obstruction theory.  Suppose that $\sN$ is smooth and that generically, $h^{-1}(\EE)=0$.  Then $\EE\to\LL_{\sM/\sN}$ is an isomorphism, and in particular, $\sM\to\sN$ is a local complete intersection morphism.
\end{lemma}

\begin{proof}
We begin by reducing to the case where $\sM\to\sN$ is an embedding of affine schemes.  It suffices to prove the lemma after a base change by a smooth presentation $V\to\sN$.  Under such a base change, $\EE\to \LL_{\sM/\sN}$ pulls back to a perfect relative obstruction theory on $\sM\times_{\sN} V \to V$.  So we may assume that $\sN=\spec S$ is an affine noetherian scheme.  Now it suffices to prove the lemma after an \'etale base change $U\to\sM$, where $U$ is an affine scheme of finite type over $S$.

Let $\iota:U\to W$ be an embedding into an affine scheme $W=\spec A$ which is smooth over $\sN$.  Let $I$ be the ideal of $U$ in $W$.  Since $\iota^*\LL_{W/\sN}$ is a vector bundle in degree $0$ and $h^0(E)\to h^0(\LL_{U/\sN})$ is an isomorphism, $\iota^*\LL_{W/\sN}\to \LL_{U/\sN}$ lifts uniquely to $\iota^*\LL_{W/\sN}\to \EE$.  Let $\FF=Cone(\iota^*\LL_{W/\sN}\to \EE)$.  Then we have a morphism of distinguished triangles:
$$\xymatrix{
\iota^*L_{W/\sN} \ar[r] \ar[d]^= & \EE \ar[r] \ar[d] & \FF \ar[r] \ar[d] & \iota^*\LL_{W/\sN}[1] \ar[d]^= \\
\iota^*\LL_{W/\sN} \ar[r] & \LL_{U/\sN} \ar[r] & \LL_{U/W} \ar[r] & \iota^*\LL_{W/\sN}[1].}$$
By taking long exact sequences, we see that $\FF$ is represented by a vector bundle in degree $-1$ which surjects onto $h^{-1}(\LL_{U/W}) = I/I^2$.  The assumption that $h^{-1}(\EE)$ is generically $0$ implies that there is a dense open subset of $U$ over which $\FF^{-1}\to I/I^2$ is an isomorphism.

By restricting to a smaller open set, we may assume that $\FF^{-1}$ is free of rank $d$.  Then a basis of $\FF$ determines elements $x_1,\ldots,x_d\in I$ which generate $I$ modulo $I^2$.  In other words $I/(x_1,\ldots,x_d)$ is generated by the image of $I^2$.  Thus $I\cdot I/(x_1,\ldots,x_d) = I/(x_1,\ldots,x_d)$ and Nakayama's lemma implies that there is an element $a\in A$ such that $a\equiv 1$ modulo $I$ and $aI\subseteq (x_1,\ldots,x_d)$ \cite[2.2]{Matsumura}.  Since $a$ does not vanish on $U$, we may invert $a$ and assume that $I=(x_1,\ldots,x_d)$.  To show that $x_1,\ldots,x_d$ is a regular sequence, it suffices to show that $\mathrm{depth}(I,A)=d$ \cite[p.131]{Matsumura}.

By assumption, $U$ has a dense open set which is a local complete intersection.  It follows that $d$ is the codimension of $U$ in $W$.  But any proper ideal $I$ of a Cohen-Macaulay ring $A$ has $\mathrm{depth}(I,A)=\mathrm{ht}(I)$ \cite[17.4]{Matsumura}, so $U$ is a local complete intersection and $I/I^2$ is free with basis $x_1,\ldots,x_d$.  This shows that $F^{-1}\to I/I^2$ is an isomorphism, which implies that $\EE\to \LL_{U/\sN}$ is an isomorphism.
\end{proof}

\section{Notation index}
\label{app:notn}

\makenotn  

\bibliographystyle{amsalpha}
\bibliography{ACW}

\end{document}